\documentclass[12pt]{article}
\title{Galois quotients of metric graphs and invariant linear systems}
\author{Song JuAe \footnote{Tokyo Metropolitan University 1-1 Minami-Ohsawa, Hachioji, Tokyo, 192-0397, Japan. E-mail: song-juae@ed.tmu.ac.jp}}
\date{}

\usepackage{amsmath,amssymb,amscd}
\usepackage{amsthm}
\usepackage{bm}
\usepackage{geometry}
\usepackage{comment}
\usepackage[pdftex]{graphicx}

\newtheorem{dfn}{Definition}[subsection]
\newtheorem{thm}[dfn]{Theorem}
\newtheorem{prop}[dfn]{Proposition}
\newtheorem{cor}[dfn]{Corollary}
\newtheorem{lemma}[dfn]{Lemma}
\newtheorem{rem}[dfn]{Remark}

\newtheorem{prb}[dfn]{Problem}

\def\Gamma{\varGamma}
\def\Z{\boldsymbol{Z}}

\begin{document}
\maketitle

\begin{abstract}
For a map $\varphi : \Gamma \rightarrow \Gamma^{\prime}$ between metric graphs and an isometric action on $\Gamma$ by finite group $K$, $\varphi$ is a {\it $K$-Galois covering} on $\Gamma^{\prime}$ if $\varphi$ is a morphism, the degree of $\varphi$ coincides with the order of $K$ and $K$ induces a transitive action on every fibre.
%For a map $\varphi : \Gamma \rightarrow \Gamma^{\prime}$ between metric graphs and an isometrical action on $\Gamma$ by finite group $K$, $\varphi$ is a {\it $K$-Galois} on $\Gamma^{\prime}$ if $\varphi$ is a morphism, the degree of $\varphi$ coincides with the order of $K$ and $K$ induces a transitive action on every fibre.
We prove that for a metric graph $\Gamma$ with an isometric action by finite group $K$, there exists a rational map, from $\Gamma$ to a tropical projective space, which induces a $K$-Galois covering on the image.
%We prove that for a metric graph $\Gamma$ with an isometrical action by finite group $K$, there exists a rational map, from $\Gamma$ to a tropical projective space, which induces a $K$-Galois covering on the image.
By using this fact, we also prove that for a hyperelliptic metric graph without one valent points and with genus at least two, the invariant linear system of the hyperelliptic involution $\iota$ of the canonical linear system, the complete linear system associated to the canonical divisor, induces an $\langle \iota \rangle$-Galois covering on a tree.
This is an analogy of the fact that a compact Riemann surface  is hyperelliptic if and only if the canonical map, the rational map induced by the canonical linear system, is a double covering on a projective line $\boldsymbol{P}^1$.
%This is an analogy of the fact that a compact Riemann surface  is hyperelliptic if and only if the canonical map, the rational map induced by the complete linear system associated to a canonical divisor, is a double covering on a projective line $\boldsymbol{P}^1$.
\end{abstract}

{\bf keywords}: metric graph, invariant linear subsystem, rational map, Galois covering, hyperelliptic metric graph, canonical map


%\thispagestyle{empty}

%\newpage

\tableofcontents

%\thispagestyle{empty}

%\newpage

%\clearpage
%\addtocounter{page}{-2}

\section{Introduction}

Tropical geometry is an algebraic geometry over tropical semifield $\boldsymbol{T} = (\boldsymbol{R} \cup \{ - \infty \}, {\rm max}, +)$.
A tropical curve is a one-dimensional object obtained from a compact Riemann surface by a limit operation called tropicalization and realized as a metric graph.
In this paper, a metric graph means a finite connected multigraph where each edge is identified with a closed segment of $\boldsymbol{T}$.
Exactly as a compact Riemann surface, concepts of a divisor, a rational function and a complete linear system etc.~are defined on a metric graph.
A morphism of metric graphs is a (finite) harmonic map.

In tropical geometry, a hyperelliptic metric graph, {\it {\it i.e.}} a metric graph with a special action by two element group is investigated in detail in \cite{Haase=Musiker=Yu}.
In this paper, we study a metric graph with an action by a finite group and develop a quotient metric graph in a tropical projective space as the image of a rational map.
Note that in this paper, we always suppose that an action by a finite group on a metric graph is isometric.

\begin{thm}[Theorem \ref{main theorem2}]
	\label{main theorem1}
Let $\Gamma$ be a metric graph and $K$ a finite group acting on $\Gamma$.
Then, there exists a rational map, from $\Gamma$ to a tropical projective space, which induces a $K$-Galois covering on the image.
\end{thm}

Here, the definition of a $K$-Galois covering on a metric graph is given as follows.

\begin{dfn}[Definition \ref{branched $K$-Galois1}]
	\label{$K$-Galois}
{\upshape
Assume that a map $\varphi : \Gamma \rightarrow \Gamma^{\prime}$ between metric graphs and an action on $\Gamma$ by $K$ are given.
Then, $\varphi$ is a {\it $K$-Galois covering} on $\Gamma^{\prime}$ if $\varphi$ is a morphism of metric graphs, the degree of $\varphi$ coincides with the order of $K$ and the action on $\Gamma$ by $K$ induces a transitive action on every fibre by $K$.
}
\end{dfn}

For the proof of theorem \ref{main theorem1}, we use the following various results.

For a divisor $D$ on a metric graph $\Gamma$, $R(D)$ denotes the set of rational functions corresponding to the complete linear system $|D|$ together with a constant function of $- \infty$ on $\Gamma$, {\it {\it i.e.}} $R(D) := \{ f \,|\, f \text{ is a rational function other than } -\infty \text{ and } D + {\rm div}(f) \text{ is effective.}\} \cup \{ - \infty \}$.
$R(D)$ becomes a tropical semimodule over $\boldsymbol{T}$ ({\cite[Lemma $4$]{Haase=Musiker=Yu}}).

\begin{thm}[{\cite[Theorem $6$]{Haase=Musiker=Yu}}]
	\label{HMY}
$R(D)$ is finitely generated.
\end{thm}

$|D|$ is also finitely generated since $|D|$ is identified with the projection of $R(D)$.
In this paper, we show the following theorem such that we consider a finite group on Theorem \ref{HMY}.

\begin{thm}[Remark \ref{$R(D)^K$ is a tropical semimodule.}, Theorem \ref{main theorem'} and Theorem \ref{main theorem''}]
%\begin{thm}[Lemma \ref{$R(D)^K$ is a tropical semimodule.}, Theorem \ref{main theorem}, Theorem \ref{main theorem'} and Theorem \ref{main theorem''}]
	\label{$R(D)^K$ is finitely generated1.}
Let $\Gamma$ be a metric graph, $K$ a finite group acting on $\Gamma$ and $D$ a $K$-invariant effective divisor on $\Gamma$.
Then, the set $R(D)^K$ consisting of all $K$-invariant rational functions in $R(D)$ becomes a tropical semimodule and is finitely generated.
\end{thm}

We can show that the set $|D|^K$ consisting of all $K$-invariant divisors in $|D|$ is identified with the projection of $R(D)^K$.
Thus, the $K$-invariant linear system $|D|^K$ is also finitely generated.
Let $\phi_{|D|^K}$ be the rational map, from $\Gamma$ to a tropical projective space, associated to $|D|^K$.
Then, the following holds.

\begin{thm}[Theorem \ref{If $K$-injective, then $K$-Galois}]
	\label{Galois}
$\phi_{|D|^K}$ induces a $K$-Galois covering on ${\rm Im}(\phi_{|D|^K})$ if and only if $\phi_{|D|^K}$ maps distinct $K$-orbits to distinct points.
\end{thm}

For each edge of ${\rm Im}(\phi_{|D|^K})$, a natural measure defined by the $\boldsymbol{Z}$-affine structure of the tropical projective space is induced.
In the proof of theorem \ref{Galois}, it is essential to show that $\phi_{|D|^K}$ is a local isometry for the edge length defined from this measure.
For the fact that generally, the rational map defined by a finite number of rational functions on a metric graph may not induce a morphism of metric graphs since the rational map may not be harmonic, Theorem \ref{Galois} states that the rational map induced by a $K$-invariant linear system induces a morphism of metric graphs if it satisfies the condition in Theorem \ref{Galois}.
Moreover, we show the following theorem.

\begin{thm}[Theorem \ref{$K$-ample1}]
	\label{$K$-ample}
There exists a $K$-invariant effective divisor $D$ on a metric graph with an action of a finite group $K$ such that $\phi_{|D|^K}$ maps distinct $K$-orbits to distinct points.
\end{thm}

In conclusion, we obtain Theorem \ref{main theorem1}.
Especially when the group $K$ is trivial, we have the following corollary.

\begin{cor}[Corollary \ref{embedded in a tropical projective space1}]
	\label{embedded in a tropical projective space}
A metric graph is embedded in a tropical projective space by a rational map.
\end{cor}

For a canonical map, which is the rational map induced by a canonical linear system, Haase--Musiker--Yu\cite{Haase=Musiker=Yu} showed the following theorem.

\begin{thm}[{\cite[Theorem $49$]{Haase=Musiker=Yu}}]
	\label{HMY2}
A metric graph whose canonical map is not injective is hyperelliptic.
\end{thm}

In the proof of Theorem \ref{HMY2}, Haase--Musiker--Yu\cite{Haase=Musiker=Yu} gave all hyperelliptic metric graphs satisfying the condition concretely.
Moreover, Haase--Musiker--Yu\cite{Haase=Musiker=Yu} showed that the inverse of Theorem \ref{HMY2} does not hold and posed the problem of other characterizing metric graphs with non-injective canonical maps.
As answers of this problem, we give Theorem \ref{application1} and Corollary \ref{application2} as follows.

\begin{thm}[Theorem \ref{canonical map1}]
	\label{application1}
Let $\Gamma$ be a metric graph without one valent points.
Then, the canonical map induces a morphism which is a double covering on the image if and only if the genus of $\Gamma$ is two.
\end{thm}

By Theorem \ref{Galois}, Theorem \ref{HMY2} and its proof and Theorem \ref{application1}, we have the following.

\begin{cor}[Corollary \ref{canonical map2}]
	\label{application2}
Let $\Gamma$ be a metric graph of genus at least three without one valent points.
Then, the canonical map of $\Gamma$ is not injective if and only if the map induced by the canonical map is not harmonic.
\end{cor}

By using Theorem \ref{main theorem1} and the fact that for the rational map induced by the canonical linear system associated with a divisor whose degree is two and whose rank is one on a metric graph, the image is a tree and the order of every fibre is one or two ({\cite[Proposition $48$]{Haase=Musiker=Yu}}), we have the following.

\begin{thm}[Theorem \ref{canonical map3}]
	\label{double covering}
For a hyperelliptic metric graph with genus at least two without one valent points, the invariant linear system of the hyperelliptic involution $\iota$ of the canonical linear system induces a rational map whose image is a tree and which is a $\langle \iota \rangle$-Galois covering on the image.
%For a hyperelliptic metric graph with genus at least two without one valent points, the invariant linear system of the hyperelliptic involution $\iota$ of the canonical linear system induces a rational map whose image is a tree and which is a branched $\langle \iota \rangle$-Galois covering on the image.
\end{thm}

Theorem \ref{double covering} means that an analogy of the fact the canonical map of a classical hyperelliptic compact Riemann surface is a double covering holds by the rational map induced not by the canonical linear system but by an invariant linear subsystem of the hyperelliptic involution of the canonical system for a hyperelliptic metric graph.

The length of each edge of ${\rm Im}(\phi_{|D|^K})$ in Theorem \ref{Galois} have not given in \cite{Haase=Musiker=Yu}.
We show this edge length is defined naturally and then we become to be able to argue whether a rational map is harmonic or not.

In this paper, we recall some basic facts corresponding to metric graphs in Section $2$.
%In this paper, we recall some basic facts of metric graphs in Section $2$.
We prove Theorem \ref{main theorem1} and corresponding statements in Section $3$.
%We prove Theorem \ref{$R(D)^K$ is finitely generated1.} in Section $3$.
Metric graph with edge-multiplicities and harmonic morphisms between them, we need in Section $3$, are defined in Section $4$.
%Section $4$ is devoted to defining a Galois covering on a metric graph and to explain how to classify Galois coverings on trees with Galois group $\boldsymbol{Z}/n\boldsymbol{Z}$ or $V$, where $n$ is a positive integer and $V$ is the Klein's four group.
%In Section $5$, we prove our main theorem (=Theorem \ref{main theorem1}) and several results corresponding to the main theorem.
%Metric graph with edge-multiplicities and harmonic morphisms between them, we need in Section $5$, are defined in Section $6$.
%In Appendix, we show the fact that for metric graphs $\Gamma$ and $\Gamma^{\prime}$, if there exists a harmonic morphism $\varphi : \Gamma \rightarrow \Gamma^{\prime}$ with non-zero degree, then the genus of $\Gamma$ is not less than the genus of $\Gamma^{\prime}$.

\section{Preliminaries}

In this section, we briefly recall some basic facts of tropical algebra (\cite{Akiba},\cite{Katakura}), metric graphs (\cite{Kawaguchi=Yamaki}), divisors on metric graphs (\cite{ABBR1}, \cite{Chan}, \cite{GK}, \cite{Kawaguchi=Yamaki}, \cite{MZ}), harmonic morphisms of metric graphs (\cite{Chan}, \cite{Haase=Musiker=Yu}, \cite{Kageyama}), and chip-firing moves on metric graphs (\cite{Haase=Musiker=Yu}), which we need later.
%In this section, we briefly recall the theories of tropical algebra (\cite{Akiba},\cite{Katakura}), metric graphs (\cite{Kawaguchi=Yamaki}), divisors on metric graphs (\cite{ABBR1}, \cite{Chan}, \cite{GK}, \cite{Kawaguchi=Yamaki}, \cite{MZ}), harmonic morphisms of metric graphs (\cite{Chan}, \cite{Haase=Musiker=Yu}, \cite{Kageyama}), and chip-firing moves on metric graphs (\cite{Haase=Musiker=Yu}), which we need later.
%In this section, we briefly recall the theories of tropical algebra (\cite{Akiba},\cite{Katakura}), metric graphs (\cite{Kawaguchi=Yamaki}), divisors on metric graphs (\cite{ABBR1}, \cite{Chan}, \cite{GK}, \cite{Kawaguchi=Yamaki}, \cite{MZ}), harmonic morphisms of metric graphs (\cite{Chan}, \cite{Haase=Musiker=Yu}, \cite{Kageyama}), chip-firing moves on metric graphs (\cite{Haase=Musiker=Yu}) and Riemann--Hurwitz formula (\cite{Baker=Norine}, \cite{Mednykh}), which we need later.

\subsection{Tropical algebra}
The set of $\boldsymbol{T}:=\boldsymbol{R} \cup \{ - \infty \}$ with two tropical operations:
\begin{center}
$a \oplus b := {\rm max}\{ a, b \}$~~~and~~~$a \odot b := a + b$,
\end{center}
where both $a$ and $b$ are in $\boldsymbol{T}$, becomes a semifield.
$\boldsymbol{T}=(\boldsymbol{T}, \oplus, \odot)$ is called the {\it tropical semifield} and $\oplus$ (resp. $\odot$) is called {\it tropical sum} (resp. {\it tropical multiplication}).
We frequently write $a \oplus b $ and $a \odot b$ as ``$a + b$'' and ``$ab$'', respectively.

A vector $\boldsymbol{v} \in \boldsymbol{T}^n$ is {\it primitive} if all coefficients of $\boldsymbol{v}$ are integers and their greatest common divisor is one.
For a vector $\boldsymbol{u} \in \boldsymbol{Q}^n$, its length is defined as $\lambda$ such that $\boldsymbol{u} = \lambda \boldsymbol{v}$, where $\boldsymbol{v} \in \boldsymbol{Z}^n$ is the primitive vector with the same direction as $\boldsymbol{u}$.
For a vector $\boldsymbol{u} = ( u_1,\ldots, u_n ) \in \boldsymbol{T}^n$, we define the length of $\boldsymbol{u}$ as $\infty$ if each $u_i \in \boldsymbol{Q} \cup \{ -\infty \}$ and some $u_j = - \infty$.
In each case, we call $\lambda$ or $\infty$ the {\it lattice length} of $\boldsymbol{u}$.

For $(x_1, \ldots, x_n) \in \boldsymbol{T}^n$ and $a \in \boldsymbol{T}$, we define a scalar operation in $\boldsymbol{T}^n$ as follows:
\[
``a(x_1, \ldots, x_n)\text{''} := (``ax_1\text{''}, \ldots, ``ax_n\text{''}).
\]
For $\boldsymbol{x}, \boldsymbol{y} \in \boldsymbol{T}^{n + 1} \setminus \{ (- \infty, \ldots, - \infty) \}$, we define the following relation $\sim$:
\begin{center}
$\boldsymbol{x} \sim \boldsymbol{y} \iff$ there exists a real number $\lambda$ such that $\boldsymbol{x} =$``$\lambda \boldsymbol{y}$''.
\end{center}
The relation $\sim$ becomes an equivalence relation.
$\boldsymbol{TP}^n := \boldsymbol{T}^{n+1} / \sim$ is called the {\it $n$-dimensional tropical projective space}.

Let $\boldsymbol{u} = ( u_1 : \cdots: u_{n+1})$ and $\boldsymbol{v} = ( v_1 : \cdots: v_{n+1})$ be distinct two points on $\boldsymbol{TP}^n (n \ge 2)$.
A distance between $\boldsymbol{u} = ( u_1 : \cdots: u_{n+1})$ and $\boldsymbol{v} = ( v_1 : \cdots: v_{n+1})$ is defined as ``the lattice length of $((u_1 - u_i) - (v_1 - v_i), \ldots, (u_n - u_i) - (v_n - v_i))$''$ = l \cdot {\rm gcd}((u_1 - u_i) - (v_1 - v_i), \ldots, (u_n - u_i) - (v_n - v_i))$ for some $i$ if all $u_j - u_i ,v_j - v_i$ are rational numbers, where $l$ is a positive rational number such that all $\frac{(u_j - u_i)}{l}$ and $\frac{(v_j - v_i)}{l}$ are integers.
A distance between a point and itself on $\boldsymbol{TP}^n$ is defined by zero.

\begin{lemma}
	\label{well-definedness of length}
Let $\boldsymbol{u} = ( u_1 : \cdots: u_{n+1})$ and $\boldsymbol{v} = ( v_1 : \cdots: v_{n+1})$ be distinct two points on $\boldsymbol{TP}^n (n \ge 2)$ such that for some $i$, all $u_j - u_i ,v_j - v_i$ are integers.
Then, 
\begin{eqnarray*}
&&{\rm gcd}((u_1 - u_i) - (v_1 - v_i), \ldots, (u_n - u_i) - (v_n - v_i))\\ &=& {\rm gcd}((u_1 - u_k) - (v_1 - v_k), \ldots, (u_n - u_k) - (v_n - v_k))
\end{eqnarray*}
holds for any $k$.
%}
\end{lemma}

\begin{proof}
Let $l_i :={\rm gcd}((u_1 - u_i) - (v_1 - v_i), \ldots, (u_n - u_i) - (v_n - v_i))$ for any $i$.
For any $k$, we have integers $m_k$ and $t_k$ such that $(u_k - u_i) - (v_k - v_i) = l_i \cdot m_k, {\rm gcd}(m_1, \ldots, m_n)=1$, $(u_k - u_j) - (v_k - v_j) = l_j \cdot t_k$ and ${\rm gcd}(t_1, \ldots, t_n)=1$.
Since $(u_k - u_i) - (v_k - v_i) = (u_k - v_k) - (u_i - v_i)$,
\begin{eqnarray*}
l_j &=& {\rm gcd}((u_1 - u_j) - (v_1 - v_j), \ldots, (u_n - u_j) - (v_n - v_j))\\
&=& {\rm gcd}((u_1 - v_1) - (u_j - v_j), \ldots, (u_n - v_n) - (u_j - v_j))\\
&=& {\rm gcd}((u_i - v_i) + l_i \cdot m_1 - (u_j - v_j), \ldots, (u_i - v_i) + l_i \cdot m_n - (u_j - v_j))\\
&=& {\rm gcd}((u_i - u_j) - (v_i - v_j) + l_i \cdot m_1, \ldots, (u_i - u_j) - (v_i - v_j) + l_i \cdot m_n)\\
&=& {\rm gcd}(l_j \cdot t_i + l_i \cdot m_1, \ldots, l_j \cdot t_i + l_i \cdot m_n).
\end{eqnarray*}
Then $l_j$ must divide $l_i \cdot m_1, \ldots, l_i \cdot m_n$.
As ${\rm gcd}(m_1, \ldots, m_n)=1$, $l_j$ divides $l_i$.
Thus $l_j \le l_i$.
The inverse inequality holds since $i$ and $j$ are arbitrary.
\end{proof}

By Lemma \ref{well-definedness of length}, the above distance between two points of $\boldsymbol{TP}^n$ satisfying the condition is well-defined.

A {\it tropical semimodule} on $\boldsymbol{T}$ is defined like a classical module on a ring.
%A {\it tropical semimodule} on $\boldsymbol{T}$ is defined similarly a classical module on a ring.
Note that a tropical semimodule on $\boldsymbol{T}$ has two tropical operations: tropical sum $\oplus$ and tropical scalar multiplication $\odot$.
%Note that a tropical semimodule on $\boldsymbol{T}$ has two tropical operations: tropical sum and tropical scalar multiplication.
Let $R$ and $R^{\prime}$ be tropical semimodules on $\boldsymbol{T}$, respectively.
A map $f : R \rightarrow R^{\prime}$ is said a {\it homomorphism} if for any $a, b \in R$ and $\lambda \in \boldsymbol{T}$, $f(a \oplus b) = f(a) \oplus f(b)$ and $f(\lambda \odot a) = \lambda \odot f(a)$ hold.
For a homomorphism $f : R \rightarrow R^{\prime}$ of tropical semimodules, $f$ is an {\it isomorphism} if there exists a homomorphism $f^{\prime} : R^{\prime} \rightarrow R$ of tropical semimodules such that $f^{\prime} \circ f = {\rm id}_R$ and $f \circ f^{\prime} = {\rm id}_{R^{\prime}}$.
Then, $f^{\prime}$ is also an isomorphism.
Two tropical semimodules $R$ and $R^{\prime}$ are {\it isomorphic} if there exists an isomorphism of tropical semimodules between them.

\subsection{Metric graphs}

In this paper, a {\it graph} means an unweighted, finite connected nonempty multigraph.
Note that we allow the existence of loops.
For a graph $G$, the sets of vertices and edges are denoted by $V(G)$ and $E(G)$, respectively.
The {\it genus} of $G$ is defined by $g(G):=|E(G)|-|V(G)|+1$.
The {\it valence} ${\rm val}(v)$ of a vertex $v$ of $G$ is the number of edges emanating from $v$, where we count each loop as two.
A vertex $v$ of $G$ is a {\it leaf end} if $v$ has valence one.
A {\it leaf edge} is an edge of $G$ adjacent to a leaf end.
%A {\it leaf edge} is an edge of $G$ having a leaf end.

An {\it edge-weighted graph} $(G, l)$ is the pair of a graph $G$ and a function $l: E(G) \rightarrow {\boldsymbol{R}}_{>0} \cup \{\infty\}$ called a {\it length function}, where $l$ can take the value $\infty$ on only leaf edges.
A {\it metric graph} is the underlying $\infty$-metric space of an edge-weighted graph $(G, l)$, where each edge of $G$ is identified with the closed interval $[0,l(e)]$ and if $l(e)=\infty$, then the leaf end of $e$ must be identified with $\infty$.
%A {\it metric graph} is the underlying $\infty$-metric space of an edge-weighted graph $(G, l)$, where each edge of $G$ is identified with the closed interval $[0,l(e)]$.
Such a leaf end identified with $\infty$ is called a {\it point at infinity} and any other point is said to be a {\it finite point}.
%For a point $x$ on a metric graph $\Gamma$ obtained from $(G, l)$, if the distances between $x$ and all points on $\Gamma$ other than $x$ are infinity, then $x$ is called a {\it point at infinity}, else, $x$ is said to be a {\it finite point}.
For the above metric graph $\Gamma$, $(G, l)$ is said to be its {\it model}.
There are many possible models for $\Gamma$.
We construct a model $(G_{\circ}, l_{\circ})$ called the {\it canonical model} of $\Gamma$ as follows.
Generally, we determine $V(G_{\circ}):= \{ x \in \Gamma~|~{\rm val}(x) \neq 2 \}$, where the {\it valence} ${\rm val}(x)$ of $x$ is the number of connected components of $U \setminus \{ x \}$ with any sufficiently small connected neighborhood $U$ of $x$ in $\Gamma$ except following two cases.
%Generally, we determine $V(G_{\circ}):= \{ x \in \Gamma~|~{\rm val}(x) \neq 2 \}$, where the valence ${\rm val}(x)$ is the number of connected components of $U \setminus \{ x \}$ with any sufficiently small connected neighborhood $U$ of $x$ in $\Gamma$ except following two cases.
%Generally, we determine $V(G_{\circ}):= \{ x \in \Gamma~|~{\rm val}(x) \neq 2 \}$, where the valence ${\rm val}(x)$ is the number of connected components of $U_x \setminus \{ x \}$ with $U_x$ being any sufficiently small connected neighborhood of $x$ in $\Gamma$ except following two cases.
When $\Gamma$ is a circle, we determine $V(G_{\circ})$ as the set consisting of one arbitrary point on $\Gamma$.
When $\Gamma$ is the $\infty$-metric space obtained from the graph consisting only of two edges with length of $\infty$ and three vertices adjacent to these edges, $V(G_{\circ})$ consists of the two endpoints of $\Gamma$ (those are points at infinity) and an any point on $\Gamma$ as the origin,
Since connected components of $\Gamma \setminus V(G_{\circ})$ consist of open intervals, whose lengths determine the length function $l_{\circ}$.
If a model $(G, l)$ of $\Gamma$ has no loops, then $(G, l)$ is said to be a {\it loopless model} of $\Gamma$.
For a model $(G, l)$ of $\Gamma$, the loopless model for $(G, l)$ is obtained by regarding all midpoints of loops of $G$ as vertices and by adding them to the set of vertices of $G$.
The loopless model for the canonical model of a metric graph is called the {\it canonical loopless model}.

For terminology, in a metric graph $\Gamma$, an edge of $\Gamma$ means an edge of the underlying graph $G_{\circ}$ of the canonical model $(G_{\circ}, l_{\circ})$.
Let $e$ be an edge of $\Gamma$ which is not a loop.
We regard $e$ as a closed subset of $\Gamma$, {\it {\it i.e.}}, including the endpoints $v_1, v_2$ of $e$.
The {\it relative interior} of $e$ is $e^{\circ} = e \setminus \{ v_1, v_2 \}$.
For a point $x$ on $\Gamma$, a connected component of $U \setminus \{ x \}$ with any sufficiently small connected neighborhood $U$ of $x$ is a {\it half-edge} of $x$.
%For a point $x$ on $\Gamma$, a connected component of $U_x \setminus \{ x \}$ with any sufficiently small connected neighborhood $U_x$ of $x$ is a {\it half-edge} of $x$.

For a model $(G, l)$ of a metric graph $\Gamma$, we frequently identify a vertex $v$ (resp. an edge $e$) of $G$ with the point corresponding to $v$ on $\Gamma$ (resp. the closed subset corresponding to $e$ of $\Gamma$).

The {\it genus} $g(\Gamma)$ of a metric graph $\Gamma$ is defined to be its first Betti number, where one can check that it is equal to $g(G)$ of any model $(G, l)$ of $\Gamma$.
A metric graph of genus zero is called a {\it tree}.

\subsection{Divisors on metric graphs}

Let $\Gamma$ be a metric graph.
An element of the free abelian group ${\rm Div}(\Gamma)$ generated by points on $\Gamma$ is called a {\it divisor} on $\Gamma$.
For a divisor $D$ on $\Gamma$, its {\it degree} ${\rm deg}(D)$ is defined by the sum of the coefficients over all points on $\Gamma$.
We write the coefficient at $x$ as $D(x)$.
A divisor $D$ on $\Gamma$ is said to be {\it effective} if $D(x) \ge 0$ for any $x$ in $\Gamma$. If $D$ is effective, we write simply $D \ge 0$.
For an effective divisor $D$ on $\Gamma$, the set of points on $\Gamma$ where the coefficient(s) of $D$ is not zero is called the {\it support} of $D$ and written as ${\rm supp}(D)$.
%The set of points on $\Gamma$ where the coefficient(s) of $D$ is not zero is called the {\it support} of $D$ and written as ${\rm supp}(D)$.
The {\it canonical divisor} $K_{\Gamma}$ of $\Gamma$ is defined as $K_{\Gamma} := \sum_{x \in \Gamma}({\rm val}(x) - 2) \cdot x$.

A {\it rational function} on $\Gamma$ is a constant function of $-\infty$ or a piecewise linear function with integer slopes and with a finite number of pieces, taking the value $\pm \infty$ only at points at infinity.
${\rm Rat}(\Gamma)$ denotes the set of rational functions on $\Gamma$.
For a point $x$ on $\Gamma$ and $f$ in ${\rm Rat}(\Gamma)$ which is not constant $-\infty$, the sum of the outgoing slopes of $f$ at $x$ is denoted by ${\rm ord}_x(f)$.
If $x$ is a point at infinity and $f$ is infinite there, we define ${\rm ord}_x(f)$ as the outgoing slope from any sufficiently small connected neighborhood of $x$.
Note when $\Gamma$ is a singleton, for any $f$ in ${\rm Rat}(\Gamma)$, we define ${\rm ord}_x(f) := 0$.
This sum is $0$ for all but finite number of points on $\Gamma$, and thus
\[
{\rm div}(f):=\sum_{x \in \Gamma}{\rm ord}_x(f) \cdot x
\]
is a divisor on $\Gamma$, which is called the {\it principal divisor} defined by $f$.
Two divisors $D$ and $E$ on $\Gamma$ are said to be {\it linearly equivalent} if $D-E$ is a principal divisor.
We handle the values $\infty$ and $-\infty$ as follows.
Let $f, g$ in ${\rm Rat}(\Gamma)$ take the value $\infty$ and $-\infty$ at a point $x$ at infinity on $\Gamma$ respectively,
and $y$ be any point in any sufficiently small neighborhood of $x$.
When ${\rm ord}_x(f) + {\rm ord}_x(g)$ is negative, then $(f \odot g)(x) := \infty$.
When ${\rm ord}_x(f) + {\rm ord}_x(g)$ is positive, then $(f \odot g)(x) := -\infty$.
Remark that the constant function of $-\infty$ on $\Gamma$ dose not determine a principal divisor.
For a divisor $D$ on $\Gamma$, the {\it complete linear system} $|D|$ is defined by the set of effective divisors on $\Gamma$ being linearly equivalent to $D$.

For a divisor $D$ on a metric graph, let $R(D)$ be the set of rational functions $f \not\equiv -\infty$ such that $D + {\rm div}(f)$ is effective together with $-\infty$.
When ${\rm deg}(D)$ is negative, $|D|$ is empty, so is $R(D)$.
Otherwise, from the argument in Section $3$ of \cite{Haase=Musiker=Yu}, $|D|$ is not empty and consequently so is $R(D)$.
Hereafter, we treat only divisors of nonnegative degree.

\begin{rem}[\cite{Song} and cf. {\cite[Lemma 4]{Haase=Musiker=Yu}}]
%\begin{lemma}[cf. {\cite[Lemma 4]{Haase=Musiker=Yu}}]
	\label{R(D) is tropical semimodule}
\upshape{
$R(D)$ becomes a tropical semimodule on $\boldsymbol{T}$ by extending above tropical operations onto functions, giving pointwise sum and product.
%\end{lemma}
}
\end{rem}

%By the definition of ${\rm ord}_x(f)$ for a point $x$ at infinity and $f$ in ${\rm Rat}(\Gamma)$, we can prove Lemma \ref{R(D) is tropical semimodule} in the same way of {\cite[Lemma 4]{Haase=Musiker=Yu}}.

For a tropical subsemimodule $M$ of $(\boldsymbol{R} \cup \{ \pm \infty \})^{\Gamma}$ (or of $\boldsymbol{R}^{\Gamma}$), $f$ in $M$ is called an {\it extremal of} $M$ when it implies $f = g_1$ or $f = g_2$ that any $g_1$ and $g_2$ in $M$ satisfies $f = g_1 \oplus g_2$.

%\begin{rem}[{\cite[Proposition 8]{Haase=Musiker=Yu}}]
%{\upshape
%Any finitely generated tropical subsemimodule $\widetilde{M}$ of $\boldsymbol{R}^{\Gamma}$ is generated by the extremals of $\widetilde{M}$.
%}
%\end{rem}

%With the adaptation for $\pm \infty$, we can prove the following lemma in same way as the above remark.

\begin{rem}[\cite{Song}]
\upshape{
%\begin{lemma}
	\label{generator is extremal}
Any finitely generated tropical subsemimodule $M$ of $R(D) \subset (\boldsymbol{R} \cup \{ \pm \infty \})^{\Gamma}$ is generated by the extremals of $M$.
%\end{lemma}
}
\end{rem}

For a divisor $D$ on a metric graph $\Gamma$, we set $r(D)$, called the {\it rank} of $D$, as the minimum integer $s$ such that for some effective divisor $E$ with degree $s - 1$, the complete linear system associated to $D - E$ is empty set.

A Riemann--Roch theorem for finite loopless graphs was established by Baker--Norine (\cite{Baker=Norine1}).
A Riemann--Roch theorem for metric graphs was proven independently by Gathmann--Kerber (\cite{GK}) and by Mikhalkin--Zharkov (\cite{MZ}). 

\begin{rem}[Riemann--Roch theorem for metric graphs]
{\upshape
Let $\Gamma$ be a metric graph and $D$ a divisor on $\Gamma$.
Then, $r(D) - r(K_{\Gamma} - D) = {\rm deg}(D) + 1 - g(\Gamma)$ holds.
}
\end{rem}

Let $\Gamma$ be a metric graph of genus at least two.
$\Gamma$ is {\it hyperelliptic} if there exists a divisor on $\Gamma$ whose degree is two and whose rank is one.
A binary group action on $\Gamma$ with a tree quotient is called a {\it hyperelliptic involution} of $\Gamma$.
Chan (\cite{Chan}), Amini--Baker--Brugall\'{e}--Rabinoff (\cite{ABBR1}) and Kawaguchi--Yamaki  (\cite{Kawaguchi=Yamaki}) investigated hyperelliptic metric graphs.

\begin{rem}[{\cite[Theorem 5]{Kawaguchi=Yamaki}}]
{\upshape
Let $\Gamma$ be a metric graph of genus at least two without one valent points.
Then, the following are equivalent:
\begin{itemize}
\item[$(1)$] $\Gamma$ is hyperelliptic;
\item[$(2)$] $\Gamma$ has a hyperelliptic involution.
\end{itemize}
Furthermore, a hyperelliptic involution is unique.
}
\end{rem}

\subsection{Harmonic morphisms}

Let $\Gamma, \Gamma^{\prime}$ be metric graphs, respectively, and $\varphi : \Gamma \rightarrow \Gamma^{\prime}$ be a continuous map.
The map $\varphi$ is called a {\it morphism} if there exist a model $(G, l)$ of $\Gamma$ and a model $(G^{\prime}, l^{\prime})$ of $\Gamma^{\prime}$ such that the image of the set of vertices of $G$ by $\varphi$ is a subset of the set of vertices of $G^{\prime}$, the inverse image of the relative interior of any edge of $G^{\prime}$ by $\varphi$ is the union of the relative interiors of a finite number of edges of $G$ and the restriction of $\varphi$ to any edge $e$ of $G$ is a dilation by some nonnegative integer factor ${\rm deg}_e(\varphi)$.
Note that the dilation factor on $e$ with ${\rm deg}_e(\varphi) \ne 0$ represents the ratio of the distance of the images of any two points $x$ and $y$ except points at infinity on $e$ to that of original $x$ and $y$.
If an edge $e$ is mapped to a vertex of $G^{\prime}$ by $\varphi$, then ${\rm deg}_e(\varphi) = 0$.
The morphism $\varphi$ is said to be {\it finite} if ${\rm deg}_e(\varphi) > 0$ for any edge $e$ of $G$.
For any half-edge $h$ of any point on $\Gamma$, we define ${\rm deg}_h(\varphi)$ as ${\rm deg}_e(\varphi)$, where $e$ is the edge of $G$ containing $h$.

Let $\Gamma^{\prime}$ be not a singleton and $x$ a point on $\Gamma$.
The morphism $\varphi$ is {\it harmonic at} $x$ if the number
\[
{\rm deg}_x(\varphi) := \sum_{x \in h \mapsto h^{\prime}}{\rm deg}_h(\varphi)
\]
is independent of the choice of half-edge $h^{\prime}$ emanating from $\varphi(x)$, where $h$ is a connected component of the inverse image of $h^{\prime}$ by $\varphi$ containing $x$. The morphism $\varphi$ is {\it harmonic} if it is harmonic at all points on $\Gamma$.
One can check that if $\varphi$ is a finite harmonic morphism, then the number
\[
{\rm deg}(\varphi) := \sum_{x \mapsto x^{\prime}}{\rm deg}_x(\varphi)
\]
is independent of the choice of a point $x^{\prime}$ on $\Gamma^{\prime}$, and is said the {\it degree} of $\varphi$, where $x$ is an element of the inverse image of $x^{\prime}$ by $\varphi$.
If $\Gamma^{\prime}$ is a singleton and $\Gamma$ is not a singleton, for any point $x$ on $\Gamma$, 
we define ${\rm deg}_x(\varphi)$ as zero so that we regard $\varphi$ as a harmonic morphism of degree zero.
If both $\Gamma$ and $\Gamma^{\prime}$ are singletons, we regard $\varphi$ as a harmonic morphism which can have any number of degree.

The collection of metric graphs together with harmonic morphisms between them forms a category.

Let $\varphi : \Gamma \rightarrow \Gamma^{\prime}$ be a finite harmonic morphism between metric graphs.
%For $f$ in ${\rm Rat}(\Gamma)$, the {\it push-forward} of $f$ is the function $\varphi_\ast f: \Gamma^{\prime} \rightarrow \boldsymbol{R} \cup \{ \pm \infty \}$ defined by 
%\[
%\varphi_\ast f(x^{\prime}) := \sum_{\substack{x \in \Gamma \\ \varphi(x) = x^{\prime}}} {\rm deg}_x(\varphi) \cdot f(x).
%\]
The {\it pull-back} of $f^{\prime}$ in ${\rm Rat}(\Gamma^{\prime})$ is the function $\varphi^{\ast}f^{\prime} : \Gamma \rightarrow \boldsymbol{R} ~\cup \{ \pm \infty \}$ defined by $\varphi^{\ast}f^{\prime} := f^{\prime} \circ \varphi$.
We define the {\it push-forward homomorphism} on divisors $\varphi_\ast : {\rm Div}(\Gamma) \rightarrow {\rm Div}(\Gamma^{\prime})$ by 
\[
\varphi_\ast (D) := \sum_{x \in \Gamma}D(x) \cdot \varphi(x).
\]
The {\it pull-back homomorphism} on divisors $\varphi^{\ast}:{\rm Div}(\Gamma^{\prime}) \rightarrow {\rm Div}(\Gamma)$ is defined to be
\[
\varphi^{\ast} (D^{\prime}) := \sum_{x \in \Gamma}{\rm deg}_x(\varphi) \cdot D^{\prime}(\varphi(x)) \cdot x.
\]
One can check that ${\rm deg}(\varphi_\ast(D)) = {\rm deg}(D)$ and $\varphi^{\ast} ({\rm div}(f^{\prime})) = {\rm div}(\varphi^{\ast} f^{\prime})$ for any divisor $D$ on $\Gamma$ and any $f^{\prime}$ in ${\rm Rat}(\Gamma^{\prime})^{\times}$, respectively (cf. \cite[Proposition 4.2]{Baker=Norine}).
%One can check that ${\rm deg}(\varphi_\ast(D)) = {\rm deg}(D)$, ${\rm deg}(\varphi^{\ast}(D^{\prime})) = {\rm deg}(\varphi) \cdot {\rm deg}(D^{\prime})$, $\varphi_\ast ({\rm div}(f)) = {\rm div}(\varphi_\ast f)$ and $\varphi^{\ast} ({\rm div}(f^{\prime})) = {\rm div}(\varphi^{\ast} f^{\prime})$ for any divisors $D$, $D^{\prime}$ on $\Gamma$, $\Gamma^{\prime}$, any $f$ in ${\rm Rat}(\Gamma)^{\times}$ and any $f^{\prime}$ in ${\rm Rat}(\Gamma^{\prime})^{\times}$, respectively (cf. \cite[Proposition 4.2]{Baker=Norine}).

\subsection{Chip-firing moves}

In \cite{Haase=Musiker=Yu}, Haase, Musiker and Yu used the term {\it subgraph} of a metric graph as a compact subset of the metric graph with a finite number of connected components and  defined the {\it chip firing move} ${\rm CF}(\widetilde{\Gamma_1}, l)$ by a subgraph $\widetilde{\Gamma_1}$ of a metric graph $\widetilde{\Gamma}$ and a positive real number $l$ as the rational function ${\rm CF}(\widetilde{\Gamma_1}, l)(x) := - {\rm min}(l, {\rm dist}(x, \widetilde{\Gamma_1}))$, where ${\rm dist}(x, \widetilde{\Gamma_1})$ is the infimum of the lengths of the shortest path to arbitrary points on $\widetilde{\Gamma_1}$ from $x$.
They proved that every rational function on a metric graph is an (ordinary) sum of chip firing moves (plus a constant) ({\cite[Lemma 2]{Haase=Musiker=Yu}}) with the concept of a {\it weighted chip firing move}.
This is a rational function on a metric graph having two disjoint proper subgraphs $\widetilde{\Gamma_1}$ and $\widetilde{\Gamma_2}$ such that the complement of the union of $\widetilde{\Gamma_1}$ and $\widetilde{\Gamma_2}$ in $\widetilde{\Gamma}$ consists only of open line segments and such that the rational function is constant on $\widetilde{\Gamma_1}$ and $\widetilde{\Gamma_2}$ and linear (smooth) with integer slopes on the complement.
A weighted chip firing move is an (ordinary) sum of chip firing moves (plus a constant) ({\cite[Lemma 1]{Haase=Musiker=Yu}}).

With unbounded edges, their definition of chip firing moves needs a little correction.
Let $\Gamma_1$ be a subgraph of a metric graph $\Gamma$ which does not have any connected components consisting only of points at infinity and $l$ a positive real number or infinity.
The {\it chip firing move} by $\Gamma_1$ and $l$ is defined as the rational function ${\rm CF}(\Gamma_1, l)(x) := - {\rm min}(l, {\rm dist}(x, \Gamma_1))$.

\begin{rem}[\cite{Song}]
\upshape{
%\begin{lemma}
A weighted chip firing move on a metric graph is a linear combination of chip firing moves having integer coefficients (plus a constant).
%}
%\end{lemma}
}
\end{rem}

%\begin{proof}[Sketch of proof.]
%We use the same notations as in their proof.
%All we have to do is to show the construction for the case with $l = \infty$.
%Especially, it is sufficient to check the case that $\Gamma_1$ consists only of points at infinity.
%Suppose that $\Gamma_1$ has only one point gives only two situations.
%Firstly, $\Gamma_2$ contains a finite point.
%Then $f$ can be written as $\pm s \cdot {\rm CF}(\Gamma_2, \infty)$ plus a constant, where $s$ is the slope of $f$ on the complement.
%Secondly, $\Gamma_2$ consists only of one point at infinity.
%Taking a finite point $x$, then $f$ can be written as $\pm s \cdot ({\rm CF}(f^{-1}([f(x), \infty]), \infty) - {\rm CF}(\{x\}, \infty))$ plus a constant with same $s$ as the first situation.

%Suppose that $\Gamma_1$ has plural points.
%$\Gamma_2$ must contain at least one finite point.
%Let $x_i$ be the intersection of $\Gamma_1$ and the closure of $L_i$.
%Note that $\Gamma_1 = \{ x_1, \ldots, x_k \}$, where $k$ is no less than two.
%With the slope $s_i$ of $f$ on $e_i := L_i \sqcup \{ x_i \}$, $f$ is $\sum_{i = 1}^{k}({\pm s_i \cdot {\rm CF}(\Gamma \setminus e_i, \infty)})$ plus a constant.
%\end{proof}

%The next lemma is proven in the same way of {\cite[Lemma 2]{Haase=Musiker=Yu}} and shows the appropriateness of this definition.

\begin{rem}[\cite{Song}]
\upshape{
%\begin{lemma}
Every rational function on a metric graph is a linear combination of chip firing moves having integer coefficients (plus a constant).
%\end{lemma}
}
\end{rem}

A point on $\Gamma$ with valence two is said to be a {\it smooth} point.
We sometimes refer to an effective divisor $D$ on $\Gamma$ as a {\it chip configuration}.
We say that a subgraph $\Gamma_1$ of $\Gamma$ can {\it fire on} $D$ if for each boundary point of $\Gamma_1$ there are at least as many chips as the number of edges pointing out of $\Gamma_1$.
A set of points on a metric graph $\Gamma$ is said to be {\it cut set} of $\Gamma$ if the complement of that set in $\Gamma$ is disconnected.

\section{Rational maps induced by $|D|^K$}
%\section{Generators of $R(D)^K$}

In this section, our main concern is the rational map induced by an invariant linear system on a metric graph with an action by a finite group $K$.
We find a condition that the rational map induces a $K$-Galois covering on the image.
%Several concepts and statements appearing in this section are based on Section $7$ of \cite{Haase=Musiker=Yu}. 

\subsection{Generators of $R(D)^K$}

In this subsection, for an effective divisor $D$ on a metric graph and a finite group $K$ acting on the metric graph, we give  two proofs, other than that in \cite{Song}, of the statement that the $K$-invariant set $R(D)^K$ of $R(D)$ is finitely generated as a tropical semimodule.
%In this section, for an effective divisor $D$ on a metric graph and a finite group $K$ acting on the metric graph, we find a generating set of the $K$-invariant set $R(D)^K$ of $R(D)$ and then, show that $R(D)^K$ is finitely generated as a tropical semimodule.
When $D$ is $K$-invariant, $R(D)/\boldsymbol{R}$ is identified with the subset $|D|^K$ of $|D|$ consisting of all $K$-invariant elements of $|D|$, so the $K$-invariant linear system $|D|^K$ is finitely generated by the generating set of $R(D)^K$ modulo tropical scaling (except by $-\infty$).

%\subsection{First proof}

\begin{rem}[{\cite[Lemma 6]{Haase=Musiker=Yu}}]
	\label{finitely generation of R(D)}
{\upshape
Let $\widetilde{\Gamma}$ be a metric graph, $\widetilde{D}$ be a divisor on $\widetilde{\Gamma}$ and $S$ be the set of rational functions $f$ in $R(\widetilde{D})$ such that the support of $\widetilde{D} + {\rm div}(f)$ does not contain any cut set of $\widetilde{\Gamma}$ consisting only of smooth points.
Then
\begin{itemize}
\item[$(1)$]
$S$ contains all the extremals of $R(\widetilde{D})$,
\item[$(2)$]
$S$ is finite modulo tropical scaling (except by $-\infty$), and
\item[$(3)$]
$S$ generates $R(\widetilde{D})$ as a tropical semimodule.
\end{itemize}
}
\end{rem}

\begin{rem}[{\cite[Theorem 14]{Haase=Musiker=Yu}}]
	\label{finitely generation of R(D)2}
{\upshape
Let $G$ be a model of $\widetilde{\Gamma}$ and let $S_G$ be the set of functions $f \in R(D)$ such that the support of $D+{\rm div}(f)$ does not contain an interior cut set ({\it {\it i.e.}} a cut set consisting of points in interior of edges in the model $G$).
Then 
\begin{itemize}
\item[$(1)$]
$S_G$ contains the set $S$ from Remark \ref{finitely generation of R(D)}, and
\item[$(2)$]
$S_G$ is finite modulo tropical scaling (except by $-\infty$).
\end{itemize}
}
\end{rem}

Though in the above remarks they assume that $R(\widetilde{D})$ is a subset of $\boldsymbol{R}^{\widetilde{\Gamma}}$, the proof is applied even in the case that $R(\widetilde{D})$ is a subset of $(\boldsymbol{R} \cup \{ \pm \infty \})^{\widetilde{\Gamma}}$ with preparations in Section 2.
Also, the above remarks throws the relation between $S$ (resp. $S_G$) and $\widetilde{D}$ into relief, hence hereafter we write $S$ (resp. $S_G$) for $\widetilde{D}$ as $S(\widetilde{D})$ (resp. $S_G(\widetilde{D}))$.

Next, for $R(D)^K$, the following holds.

\begin{rem}[\cite{Song}]
\upshape{
%\begin{lemma}
	\label{$R(D)^K$ is a tropical semimodule.}
$R(D)^K$ is a tropical semimodule.
%\end{lemma}
}
\end{rem}

%\begin{proof}
%Let $c$ be in $\boldsymbol{T}$, $f, g$ in $R(D)^K$ and $\sigma$ in $K$.
%Since $R(D)$ is a tropical semimodule by Lemma \ref{R(D) is tropical semimodule}, $c \odot f$ and $f \oplus g$ are in $R(D)$.
%It is obvious that $\odot$ and $\circ$ are associative and that $\circ$ is distributive over $\oplus$ from right, both $(c \odot f) \circ \sigma$ and $(f \oplus g) \circ \sigma$ are in $R(D)^K$.
%\end{proof}

Note that $R(D + {\rm div}(f))^K = R(D)^K \odot (-f)$ for any $K$-invariant rational function $f$.

The following lemma is an extension of {\cite[Lemma 5]{Haase=Musiker=Yu}}.

\begin{rem}[\cite{Song}]
%\begin{lemma}
	\label{condition to be extremal}
\upshape{
Let $f$ be in ${\rm Rat}(\Gamma)$.
Then, $f$ is an extremal of $R(D)^K$ if and only if there are not two proper $K$-invariant subgraphs $\Gamma_1$ and $\Gamma_2$ covering $\Gamma$ such that each can fire on $D + {\rm div}(f)$.
%\end{lemma}
}
\end{rem}

If $D$ is $K$-invariant, $R(D)^K/\boldsymbol{R}$ is naturally identified with the subset $|D|^K$ of $|D|$ consisting of all $K$-invariant elements of $|D|$.
In fact, let $D$ be a $K$-invariant effective divisor on $\Gamma$. 
For any $D^{\prime} \in |D|^K$, there exists $f \in R(D)^K$ such that $D^{\prime}=D + {\rm div}(f)$.
Since both $D$ and $D^{\prime}$ are $K$-invariant, $D^{\prime} = D + {\rm div}(f \circ \sigma)$ for any $\sigma \in K$.
Thus $0 = {\rm div}(f) - {\rm div}(f \circ \sigma) = {\rm div}(f - f \circ \sigma)$ and there exists $c \in \boldsymbol{R}$ such that $f - f \circ \sigma = c$, {\it {\it i.e.}} $f = c \odot f \circ \sigma$ by Liouville's theorem.
Since the order $k$ of $\sigma$ is finite, $f = c \odot \cdots \odot c \odot f$ holds, where c is multiplied $k$ times.
As $k$ is not zero, $c$ must be zero.
Therefore $f$ is $K$-invariant and then $f \in R(D)^K$.
Conversely, $[g] \in R(D)^K / \boldsymbol{R}$ corresponds to an element $D + {\rm div}(g)$ in $|D|^K$.
%In conclusion, the following holds from Theorem \ref{main theorem}.

%\begin{thm}
%Let $\Gamma$ be a metric graph, $D$ an effective divisor on $\Gamma$ and $K$ a finite group acting on $\Gamma$.
%If $D$ is $K$-invariant, then the $K$-invariant linear subsystem $|D|^K$ of $|D|$ is finitely generated by $S(D)_K/\boldsymbol{R}$.
%\end{thm}

%\subsection{Second and third proofs}

Let $\Gamma$ be a metric graph, $K$ a finite group acting on $\Gamma$ and $D$ an effective divisor on $\Gamma$.
In this subsection, we prove that $R(D)^K$ is finitely generated as a tropical semimodule in two different ways from that in Subsection $4.1$.
In one of them, we return arguments about effective divisors to ones about $K$-invariant divisors to use the condition of generators of $R(D)$ found by Haase, Musiker and Yu.
In the other, we find generators of $R(D)^K$ by an algebraic way.
By later proof, we know that the number of generators of $R(D)^K$ is not greater than that of $R(D)$.

Let $D_1$ be the maximum $K$-invariant part of $D$, {\it {\it i.e.}} $D_1 := \sum_{x \in \Gamma}{\rm min}_{x^{\prime} \in Kx}\{D(x^{\prime})\} \cdot x$.
By the definition, both $D_1$ and the $K$-variant part $D_2 := D - D_1$ are effective.

\begin{lemma}
	\label{important lemma}
$R(D)^K = R(D_1)^K$ holds.
\end{lemma}

\begin{proof}
$D_1 + {\rm div}(f_1) \ge 0$ holds for any element $f_1$ of $R(D_1)^K$.
Therefore $D + {\rm div}(f_1) = D_2 + (D_1 + {\rm div}(f_1)) \ge 0$, {\it {\it i.e.}} $f_1 \in R(D_1)^K$.

For arbitrary element $f$ of $R(D)^K$, the set of poles of $f$ is $K$-invariant as $f$ is $K$-invariant and the set is contained in the support of $D_1$.
This means $D_1 + {\rm div}(f) \ge 0$.
In fact, if $D_1 + {\rm div}(f) < 0$ holds, then there exists a point $x$ on $\Gamma$ whose orbit by $K$ is a subset of ${\rm supp}(D_1 + {\rm div}(f))$.
Therefore $0 > D_2 + (D_1 + {\rm div}(f)) = D + {\rm div}(f)$ holds and this contradicts to $f \in R(D)^K$.
\end{proof}

By Lemma \ref{important lemma}, we can prove $(1) , (3)$ of the following theorem in the same way of the proof of {\cite[Lemma 6]{Haase=Musiker=Yu}} and $(2)$ clearly holds by {\cite[Lemma 6]{Haase=Musiker=Yu}}.

\begin{thm}
	\label{main theorem'}
In the above situation, the following hold\,:
\begin{itemize}
\item[$(1)$]
$S_G(D_1)^K$ contains all the extremals of $R(D_1)^K$,

\item[$(2)$]
$S_G(D_1)^K$ is finite modulo tropical scaling (except by $-\infty)$, and

\item[$(3)$]
$S_G(D_1)^K$ generates $R(D_1)^K$ as a tropical semimodule.
\end{itemize}
\end{thm}

The following theorem is proved with a purely algebraic way from stacked point of view.

\begin{thm}
	\label{main theorem''}
Let $\Gamma$ be a metric graph, $K$ be a finite group acting on $\Gamma$ and $D$ a $K$-invariant effective divisor on $\Gamma$.
For a minimal generating set $\{ f_1, \ldots, f_n \}$ of $R(D)$, $\{ g_1, \ldots, g_n \}$ is a generating set of $R(D)^K$, where $g_i := ``\sum_{\sigma \in K}f_i \circ \sigma\text{''}$.
\end{thm}

\begin{proof}
For any $i$ and $\tau \in K$,
\[
g_i \circ \tau = `` \left( \sum_{\sigma \in K}f_i \circ \sigma \right) \circ \tau\text{''} = ``\sum_{\sigma \in K}f_i \circ (\sigma \circ \tau)\text{''} = ``\sum_{\sigma \in K}f_i \circ \sigma\text{''} = g_i
\]
and
\[
0 \leq \tau (D + {\rm div}(f_i)) = \tau(D) + {\rm div}(f_i \circ \tau) = D + {\rm div}(f_i \circ \tau)
\]
hold.
Thus, each $g_i$ is in $R(D)^K$.
For any element $g=``\sum_{i=1}^{n}a_i f_i\text{''} \in R(D)^K$,
\[
g = ``\sum_{\sigma \in K}g \circ \sigma\text{''} = ``\sum_{\sigma \in K} \left( \sum_{i=1}^{n} a_i f_i \right) \circ \sigma\text{''} = ``\sum_{i=1}^{n} a_i \left( \sum_{\sigma \in K} f_i \circ \sigma \right) \text{''} = ``\sum_{i=1}^{n} a_i g_i\text{''}.
\]
Hence $\{ g_1, \ldots, g_n \}$ generates $R(D)^K$.
\end{proof}

\begin{cor}
Let $\Gamma$ be a metric graph, $K$ be a finite group acting on $\Gamma$ and $D$ an effective divisor on $\Gamma$.
For a minimal generating set $\{ f_1, \ldots, f_n \}$ of $R(D_1)$, $\{ g_1, \ldots, g_n \}$ is a generating set of $R(D)^K=R(D_1)^K$, where $g_i := ``\sum_{\sigma \in K}f_i \circ \sigma\text{''}$.
\end{cor}

\begin{rem}
{\upshape
$\{ g_1, \ldots, g_n \}$ is not always minimal.
Using Lemma 	\ref{condition to be extremal}, we can obtain a minimal generating set by omitting elements not being extremal from $\{ g_1 , \ldots, g_n \}$.
}
\end{rem}

%\section{Galois coverings on metric graphs}

%This section is devoted to defining a Galois covering on a metric graph and to classify Galois coverings with degree $n$ cyclic groups $\boldsymbol{Z}/n\Z$ or Klein's four group $V$ as Galois group, where $n$ is a positive integer.

\subsection{Galois covering on metric graphs}

Let $\Gamma$ be a metric graph and $K$ a finite group acting on $\Gamma$.
We define $V_1(\Gamma)$ as the set of points $x$ on $\Gamma$ such that there exists a point $y$ in any neighborhood of $x$ whose stabilizer is not equal to that of $x$.

\begin{rem}[\cite{Song}]
\upshape{
%\begin{lemma}
	\label{$V_1$ is finite}
$V_1(\Gamma)$ is a finite set.
%\end{lemma}
}
\end{rem}

We set $(G_0, l_0)$ as the canonical loopless model of $\Gamma$.
By Lemma \ref{$V_1$ is finite}, we obtain the model $(\widetilde{G_1}, \widetilde{l_1})$ of $\Gamma$ by setting the $K$-orbit of the union of $V(G_0)$ and $V_1(\Gamma)$ as the set of vertices $V(\widetilde{G_1})$.
Naturally, we can regard that $K$ acts on $V(\widetilde{G_1})$ and also on $E(\widetilde{G_1})$.
Thus, the sets $V(\widetilde{G^{\prime}})$ and $E(\widetilde{G^{\prime}})$ are defined as the quotient sets of $V(\widetilde{G_1})$ and $E(\widetilde{G_1})$ by $K$, respectively.
Let $\widetilde{G^{\prime}}$ be the graph obtained by setting  $V(\widetilde{G^{\prime}})$ as the set of vertices and $E(\widetilde{G^{\prime}})$ as the set of edges.
Since $\widetilde{G_1}$ is connected, $\widetilde{G^{\prime}}$ is also connected.
We obtain the loopless graph $G^{\prime}$ from $\widetilde{G^{\prime}}$ and the loopless model $(G_1, l_1)$ of $\Gamma$ from the inverse image of $V(G^{\prime})$ by the map defined by $K$.
Note that $V(G_1)$ contains $V(\widetilde{G_1})$.
Since $K$ is a finite group acting on $\Gamma$, the length function $l^{\prime} : E(G^{\prime}) \rightarrow {\boldsymbol{R}}_{> 0} \cup \{ \infty \}$, $[e] \mapsto |K_e| \cdot l_1(e)$ is well-defined, where $[e]$ and $K_e$ mean the equivalence class of $e$ and the stabilizer of $e$, respectively.
Let $\Gamma^{\prime}$ be the metric graph obtained from $(G^{\prime}, l^{\prime})$.
Then, $\Gamma^{\prime}$ is the quotient metric graph of $\Gamma$ by $K$.

For any edge $e$ of $G_1$, by the Orbit-Stabilizer formula, $|K_e|$ is a positive integer.
Thus, for $(G_1, l_1)$ and $(G^{\prime}, l^{\prime})$, there exists only one morphism $\pi : \Gamma \rightarrow \Gamma^{\prime}$ that satisfies ${\rm deg}_e(\pi) = |K_e|$ for any edge $e$ of $G_1$.
%Thus, for $(G_1, l_1)$ and $(G^{\prime}, l^{\prime})$, there exists only one morphism $\varphi : \Gamma \rightarrow \Gamma^{\prime}$ that satisfies ${\rm deg}_e(\varphi) = |K_e|$ for any edge $e$ of $G_1$.

%We obtain the following lemma as an extension of {\cite[Lemma 2.2]{Chan}}.

\begin{rem}[\cite{Song}]
%\begin{lemma}
\upshape{
$\pi$ is a finite harmonic morphism of degree $|K|$.
%$\varphi$ is a finite harmonic morphism of degree $|K|$.
%\end{lemma}
}
\end{rem}

Note that whether $\Gamma$ is a singleton or not agrees with whether $\Gamma^{\prime}$ is a singleton.

Let $\varphi : \Gamma \rightarrow \Gamma^{\prime}$ be a finite harmonic morphism of metric graphs.
We write the isometry transformation group of $\Gamma$ as ${\rm Isom}(\Gamma)$, {\it i.e.} ${\rm Isom}(\Gamma) := \{ \sigma : \Gamma \rightarrow \Gamma \,|\, \sigma \text{ is an isometry} \}$.

\begin{dfn}
	\label{branched $K$-Galois1}
{\upshape
Assume that a map $\varphi : \Gamma \rightarrow \Gamma^{\prime}$ between metric graphs and an action on $\Gamma$ by $K$ are given.
Then, $\varphi$ is a {\it $K$-Galois covering} on $\Gamma^{\prime}$ if $\varphi$ is a harmonic morphism of metric graphs, the degree of $\varphi$ is coincident with the order of $K$ and the action on $\Gamma$ by $K$ induces a transitive action on every fibre by $K$.
$K$ is called the {\it Galois group} of $\varphi$.
}
\end{dfn}

If $\varphi$ is a $K$-Galois covering, then $\varphi$ is finite since $K$ transitively acts on ever fibre.

%\begin{rem}
%\upshape{
%By the definition, a $K$-Galois covering $\varphi : \Gamma \rightarrow \Gamma^{\prime}$ of metric graphs is Galois.
%}
%\end{rem}

\begin{rem}
{\upshape
A $K$-Galois covering can be $K^{\prime}$-Galois for a finite group $K^{\prime}$ which is not conjugate to $K$.
}
\end{rem}

%\begin{dfn}
%{\upshape
%Let $K$ be a finite group and $\varphi : \Gamma \rightarrow \Gamma^{\prime}$ a $K$-Galois covering of metric graphs.
%A point $x$ on $\Gamma$ is a {\it ramification point} of $\varphi$ if $|Kx|$ is less than ${\rm deg}(\varphi)$ and $x$ is not smooth.
%The image of a ramification point $x$ on $\Gamma$ by $\varphi$ is said a {\it branch point} of $\Gamma$.
%}
%\end{dfn}

\begin{lemma}
	\label{There exists an isomorphism.}
There exists a finite harmonic morphism of degree one ({\it i.e.} an isomorphism) $\psi$ from the quotient metric graph $\Gamma/K$ to $\Gamma^{\prime}$ which satisfies $\varphi = \psi \circ \pi$, where $\pi : \Gamma \rightarrow \Gamma / K$ is the natural surjection.
%There exists a finite harmonic morphism of degree one ({\it i.e.} an isomorphism) $\psi$ from the quotient metric graph $\Gamma/K$ to $\Gamma^{\prime}$ which satisfies $\varphi = \psi \circ \pi$, where $\pi : \Gamma \rightarrow \Gamma / K$ is the natural surjection given in Subsection $3.1$ (written as ``$\varphi$'').
\end{lemma}

\begin{proof}
Let $(G, l), (G^{\prime}, l^{\prime})$ and $(G^{\prime \prime}, l^{\prime \prime})$ be models of $\Gamma, \Gamma^{\prime}$ and $\Gamma / K$ corresponding to $\varphi$ and $\pi$, respectively.
Let $\psi : \Gamma / K \rightarrow \Gamma^{\prime}$ be a map defined by $[x] \mapsto \varphi(x)$.
Since $[x] = Kx \subset \Gamma$ and $\varphi$ is $K$-Galois, $\varphi(Kx) = \varphi(x)$ holds.
Thus $\psi$ is well-defined.
By the definition of $\psi$, $\varphi = \psi \circ \pi$ holds.
As $\varphi$ is continuous, so is $\psi$.
For any edge $e \in E(G)$, since $\varphi$ is $K$-Galois,
\[
{\rm deg}(\varphi) = \sum_{\substack{e_1 \subset Ke \\ e_1 \in E(G)}} {\rm deg}_{e_1}(\varphi) = \sum_{\substack{e_1 \subset Ke \\ e_1 \in E(G)}} {\rm deg}_e(\varphi) = |Ke| \cdot {\rm deg}_e(\varphi).
\]
Therefore,
\[
{\rm deg}_e(\varphi) = \frac{{\rm deg}(\varphi)}{|Ke|} = \frac{|K|}{|Ke|} = |K_e|.
\]
For any edge $[e] \in E(G^{\prime \prime})$,
\[
\frac{l^{\prime \prime}([e])}{l^{\prime}(\varphi(e))} = \frac{l(e) \cdot {\rm deg}_e(\pi)}{l(e) \cdot {\rm deg}_e(\varphi)} = \frac{l(e) \cdot |K_e|}{l(e) \cdot |K_e|} = 1.
\]
Therefore, $\psi$ is a finite harmonic morphism of degree one.
\end{proof}

Note that $\pi$ in Lemma \ref{There exists an isomorphism.} is a $K$-Galois covering.

\subsection{Rational maps induced by $|D|^K$}

Several concepts and statements appearing in this subsection are based on Section $7$ of \cite{Haase=Musiker=Yu}. 

Let $\Gamma$ be a metric graph, $K$ a finite group acting on $\Gamma$ and let $D$ be a $K$-invariant effective divisor on $\Gamma$.
For a finite generating set ${\it F}=\{ f_1 , \ldots, f_n \}$ of $R(D)^K$, let $(G, l)$ be the model of $\Gamma$ such that $V(G) := V(G_0) \cup \bigcup_{i=1}^n ({\rm supp}(D + {\rm div}(f_i)))$, where $G_0$ is the underlying graph structure of the canonical loopless model $(G_0, l_0)$ of $\Gamma$.
$\phi_{\it F} : \Gamma \rightarrow \boldsymbol{TP^{n-1}}, x \mapsto (f_1 (x): \cdots: f_n (x))$ denotes the {\it rational map} induced by ${\it F}$.

\begin{prop}
	\label{image is a metric graph}
${\rm Im}(\phi_{\it F})$ is a metric graph in $\boldsymbol{TP}^{n-1}$.
\end{prop}

\begin{proof}
If $n=1$, then $\phi_{\it F}$ is a constant map from $\Gamma$ to $\boldsymbol{TP}^0$ and ${\rm Im}(\phi_{\it F})$ is a metric graph.

Let us assume that $n \ge 2$.
As $R(D)^K$ contains all constant functions, $\phi_{\it F}$ is well-defined.
Since all $f_i$s are $\boldsymbol{Z}$-affine function, the image of $\phi_{\it F}$ is a one-dimensional polyhedral complex.

Let $e=v_1 v_2$ be an edge of $G$.
As each $f_i$ is constant on $e$, $\phi_{\it F}(e)$ is a segment or a point in $\boldsymbol{TP}^{n-1}$ and the distant between $\phi_{\it F}(v_1)$ and $\phi_{\it F}(v_2)$ can be measured by the definition of rational functions on metric graphs.
Hence ${\rm Im}(\phi_{\it F})$ becomes a metric graph.
\end{proof}

When does $\phi_{\it F}$ induce a finite harmonic morphism from $\Gamma$ to ${\rm Im}(\phi_{\it F})$?
Moreover, when does $\phi_{\it F}$ induce a $K$-Galois covering on ${\rm Im}(\phi_{\it F})$?
We consider an answer for these questions.

\begin{rem}
	\label{homomorphism remark}
{\upshape
$R(D)^K$ is isomorphic to $R(D^{\prime})^K$ as a tropical semimodule for any element $D^{\prime}$ of $|D|^K$.
In fact, since $D^{\prime}$ is linearly equivalent to $D$ and both are $K$-invariant, there exists $f \in R(D)^K$ such that $D^{\prime} = D + {\rm div}(f)$ (see Subsection $3.1$).
$\psi : R(D)^K \rightarrow R(D^{\prime})^K$, $g \mapsto g - f$ and $\psi^{\prime} : R(D^{\prime})^K \rightarrow R(D)^K$, $g \mapsto g + f$ are homomorphisms of tropical semimodules and are inverses of each other.
}
\end{rem}

\begin{dfn}
{\upshape
Let $\Gamma$ be a metric graph, $K$ a finite group acting on $\Gamma$ and let $D$ be a $K$-invariant effective divisor on $\Gamma$.
$D$ is {\it $K$-very ample} if for any elements $x$ and $x^{\prime}$ of $\Gamma$ whose orbits by $K$ differ from each other, there exist $f$ and $f^{\prime}$ in $R(D)^K$ such that $f(x) - f(x^{\prime}) \not = f^{\prime}(x) - f^{\prime}(x^{\prime})$.
We call $D$ {\it $K$-ample} if some positive multiple $kD$ is $K$-very ample. 
When $K$ is trivial, we use words ``very ample'' or ``ample'' simply. 
}
\end{dfn}

\begin{rem}
{\upshape
$D$ is $K$-very ample if and only if $D^{\prime}$ is $K$-very ample for any element $D^{\prime}$ of $|D|^K$.
In fact, if $D$ is $K$-very ample, for any points $x$ and $x^{\prime}$ on $\Gamma$ whose orbits by $K$ differ from each other, there exist $g$ and $g^{\prime}$ in $R(D)^K$ such that $g(x) - g(x^{\prime}) \not= g^{\prime}(x) - g^{\prime}(x^{\prime})$.
Using $\psi$ given in Remark \ref{homomorphism remark}, $(g - f)(x) - (g - f)(x^{\prime}) = (g(x) - g(x^{\prime})) - (f(x) - f(x^{\prime})) \not= (g(x) - g(x^{\prime})) - (f^{\prime}(x) - f^{\prime}(x^{\prime})) = (g - f^{\prime})(x) - (g - f^{\prime})(x^{\prime})$.
Thus, $D^{\prime}$ is $K$-very ample.
The converse is shown in the same way.
}
\end{rem}

\begin{dfn}
{\upshape
Let ${\it F} = \{ f_1, \ldots, f_n\}$ be a finite generating set of $R(D)^K$.
$\phi_{\it F}$ is {\it $K$-injective} if $\phi_{\it F}$ separates different $K$-orbits on $\Gamma$, {\it {\it i.e.}} for any $x$ and $x^{\prime}$ in $\Gamma$ whose $K$-orbits differ from each other, $\phi_{\it F} (x) \not = \phi_{\it F} (x^{\prime})$ holds.
}
\end{dfn}

\begin{rem}
	\label{independent remark}
	\upshape{
Let ${\it F}_1 = \{ f_1, \ldots, f_n \}$ and ${\it F}_2 = \{g_1, \ldots, g_n \}$ be minimal generating sets of $R(D)^K$, respectively.
Since both ${\it F}_1$ and ${\it F}_2$ are minimal, each $g_i$ is written as $a_{i} \odot f_i$ with some real number $a_i$ by changing numbers if we need.
Thus, we can move ${\rm Im}(\phi_{{\it F}_1})$ to ${\rm Im}(\phi_{{\it F}_2})$ by the translation $(x_1 : \cdots : x_n) \mapsto (x_1 + a_1 : \cdots : x_n + a_n)$.
Hence $\phi_{{\it F}_1}$ is $K$-injective if and only if $\phi_{{\it F}_2}$ is $K$-injective.
}
\end{rem}

\begin{lemma}
$D$ is $K$-very ample if and only if the rational map associated to any finite generating set is $K$-injective.
\end{lemma}

\begin{proof}
(``if'' part)
Let ${\it F} = \{ f_1, \ldots, f_n \}$ be a generating set of $R(D)^K$.
Assume that $\phi_{\it F}$ is $K$-injective, {\it {\it i.e.}} for any $Kx \not= Kx^{\prime}$ , $\phi_{\it F}(x) \not= \phi_{\it F}(x^{\prime})$.
If for any $i$ and $j$, $f_i(x) - f_i(x^{\prime})=f_j(x) - f_j(x^{\prime})=:c$, then
\begin{eqnarray*}
\phi_{\it F}(x) &=& (f_1(x): \cdots: f_n(x)) = (f_1(x^{\prime}) + c: \cdots: f_n(x^{\prime}) + c)\\
&=& (f_1(x^{\prime}): \cdots: f_n(x^{\prime})) = \phi_{\it F}(x^{\prime}).
\end{eqnarray*}
This is a contradiction.
Thus there exist $i \not= j$ such that $f_i(x) - f_i(x^{\prime}) \not= f_j(x) - f_j(x^{\prime})$.

(``only if'' part)
Suppose that there exists a finite generating set ${\it F} = \{ f_1, \ldots, f_n \} \subset R(D)^K$ such that $\phi_{\it F}$ is not $K$-injective.
There exist distinct points $x$ and $x^{\prime}$ on $\Gamma$ whose $K$-orbits are different from each other and whose images by $\phi_{\it F}$ are same.
Therefore, there exists a real number $c$ such that $f_i(x^{\prime}) + c = f_i(x)$ for any $i$.
This means that $f_i(x) - f_i(x^{\prime}) = f_j(x) - f_j(x^{\prime})$ for any $i$ and $j$.
Hence for any $f$ and $f^{\prime}$ in $R(D)^K$, $f(x)-f(x^{\prime}) = f^{\prime}(x) - f^{\prime}(x^{\prime})$ since ${\it F}$ generates $R(D)^K$ as a tropical semimodule.
\end{proof}

If the induced rational map associated to a minimal generating set of $R(D)^K$ is $K$-injective, then $D$ is $K$-very ample since all generating sets of $R(D)^K$ contains a minimal generating set of $R(D)^K$ consisting only of extremals of $R(D)^K$ (see \cite{Song}).
%If the induced rational map associated to a minimal generating set of $R(D)^K$ is $K$-injective, then $D$ is $K$-very ample since all generating sets of $R(D)^K$ contains a minimal generating set of $R(D)^K$ consisting only of extremals of $R(D)^K$ (see Corollary \ref{unique corollary}).
Therefore, we obtain the following corollary.

\begin{cor}
$D$ is $K$-very ample if and only if the rational map associated to a minimal generating set of $R(D)^K$ is $K$-injective.
\end{cor}

\begin{lemma}
	\label{slope one lemma}
If $\Gamma$ dose not consist only of one point and for any point $x$ on $\Gamma$, there exists $f$ in $R(D)^K$ such that the support of $D + {\rm div}(f)$ contains the orbit of $x$ by $K$, then for any edge $e$ of $G$, there exists $f_i$ which has slope one on $e$.
\end{lemma}

\begin{proof}
Suppose that there exists an edge $e=v_1v_2 \in E(G)$ such that any $f_i$ dose not have slope one on $e$.
By assumptions, there exists $f_j$ has slope at least two on $e$.
By changing numbers if we need, we may assume that $f_j(v_1) > f_j(v_2)$.
Let $(f_j)_{ \ge f_i(t)}(x):={\rm max}\{ f_j(x), f_j(t) \}$ for any $t \in \Gamma$.
Since both $f_j$ and the constant $f_j(t)$ function are in $R(D)^K$ and $(f_j)_{ \ge f_i(t)}$ is the tropical sum of them, $(f_j)_{ \ge f_i(t)} \in R(D)^K$ holds.
$\Gamma_{v_1} := \{ x \in \Gamma \,|\, f_j(x) \ge f_j(v_1) \}$ is $K$-invariant and can fire on $D + {\rm div}((f_j)_{ \ge f_i(v_1)})$.
In fact, for any $x \in \Gamma_{v_1}$ and $\sigma \in K$, since $f_j(\sigma(x)) = f_j(x) \ge f_j(v_1)$ holds, $\Gamma_{v_1}$ is $K$-invariant.
For any sufficiently close point $t$ to $v_1$ on $e$ and $\Gamma_t := \{ x \in \Gamma \,|\, f_j(x) \ge f_j(t) \}$, $g := (f_j)_{ \ge f_i(t)}-(f_j)_{ \ge f_i(v_1)}$ has a constant integer slope which is different from zero on any closure of connected component of $\Gamma_t \setminus {\Gamma_{v_1}}$.
Therefore for any point $x$ on the boundary set of $\Gamma_{v_1}$ and any positive number $l$ less than the minimum of lengths of these closures,
\begin{eqnarray*}
(D + {\rm div}((f_j)_{ \ge f_j(v_1)} + ({\rm CF}(\Gamma_{v_1}, l))))(x) &\ge& (D + {\rm div}((f_j)_{ \ge f_j(v_1)} + g))(x)\\
&=& (D + {\rm div}(f_j)_{\ge f_i(t)})(x) \ge 0.
\end{eqnarray*}
Thus $(f_j)_{\ge f_j(v_1)} + {\rm CF}(\Gamma_{v_1}, l)$ is in $R(D)^K$ and has slope one on $[t, v_1] \subset e$.
This is a contradiction.
\end{proof}

\begin{rem}
{\upshape
When $K$ is trivial, the condition ``for any point $x$ on $\Gamma$, there exists $f$ in $R(D)^K$ such that the support of $D + {\rm div}(f)$ contains the orbit of $x$ by $K$'' means that the rank $r(D)$ of $D$ is greater than or equal to one. 
}
\end{rem}

\begin{lemma}
	\label{K-orbit lemma}
If $\phi_{\it F}$ is $K$-injective, then for any $x \in \Gamma$, there exists $f \in R(D)^K$ such that ${\rm supp}(D + {\rm div}(f)) \supset Kx$.
\end{lemma}

\begin{proof}
If $\Gamma$ consists only of one point $p$ , then $D$ must be of the form $kp$ with a positive integer $k \in \boldsymbol{Z}_{>0}$ and $R(D)^K = R(D)$ consists only of constant functions on $\Gamma$.
Therefore for any $f \in R(D)^K$, the support of $D + {\rm div}(f)$ coincides with the support of $D$ and it is $\{ p \}$.

Let us assume that $\Gamma$ does not consist only of one point.
Since $\Gamma$ is connected, $\Gamma$ contains a closed segment.
We show the contraposition.
Suppose that there exists a point $x$ on $\Gamma$ such that for any $f \in R(D)^K$, the support of $D + {\rm div}(f)$ does not contain $Kx$.
In particular, for any $i$, the support of $D + {\rm div}((f_i)_{\ge f_i (x)})$ (resp. the support of $D + {\rm div}(f_i)$) does not contain $Kx$, where $ (f_i)_{\ge f_i (x)}(t) := {\rm max}\{ f_i (x), f_i(t) \}$ and it is in $R(D)^K$ as it is the tropical sum of $f_i$ and the constant $f_i (t)$ function for any $t \in \Gamma$.
Therefore $(D + {\rm div}((f_i)_{\ge f_i(x)}))(x)=0$ (resp. $(D + {\rm div}(f_i))(x) = 0$) holds.
By the definition of $(f_i)_{\ge f_i(x)}, D(x) \ge 0$ and $({\rm div}((f_i)_{\ge f_i(x)}))(x) \ge 0$.
Thus $D(x) = ({\rm div}((f_i)_{\ge f_i(x)}))(x) = 0$, and then $({\rm div}(f_i))(x)=0$.
If $f_i$ is nonconstant around $x$, then there must exist a direction on which $f_i$ has positive slope and another direction on which $f_i$ has negative slope at $x$.
This means that $({\rm div}((f_i)_{\ge f_i(x)}))(x) \ge 1$ and this is a contradiction.
Consequently, $f_i$ is locally constant at $x$.
As ${\it F}$ is finite, we can choose a connected neighborhood $U_x$ of $x$ such that $\phi_{\it F}(U_x) = \phi_{\it F}(x)$.
Since $K$ is finite, $\phi_{\it F}$ is not $K$-injective.
\end{proof}

\begin{rem}
	\label{slope zero remark}
{\upshape
Since $D$ is effective, $R(D)^K$ contains all constant functions on $\Gamma$.
Therefore for any edge $e$ of $G$, there exists $f_i$ which has slope zero on $e$.}
\end{rem}

\begin{rem}
\upshape{
In Section $6$, we define metric graphs with edge-multiplicities and harmonic morphisms between them.
Hereafter, we use these concepts and so we recommend seeing Section $6$.
}
\end{rem}

\begin{thm}
	\label{If $K$-injective, then $K$-Galois}
If $\phi_{\it F}$ is $K$-injective, then $\phi_{\it F}$ induces a $K$-Galois covering on ${\rm Im}(\phi_{\it F})$ with some edge-multiplicities.
\end{thm}

\begin{proof}
If $\Gamma$ is a singleton, then the image of $\phi_{\it F}$ is also a singleton.
Since $\phi_{\it F}$ induces a finite harmonic morphism between singletons, then it is $K$-Galois.

Assume that $\Gamma$ is not a singleton. 
By Proposition \ref{image is a metric graph}, Lemma \ref{slope one lemma}, Lemma \ref{K-orbit lemma} and Remark \ref{slope zero remark}, $\phi_{\it F}$ is a local isometry.
In fact, for any edge $e = v_1 v_2$ of $G$, 
\begin{eqnarray*}
\phi_{\it F}(v_2) &=& (f_1(v_2) : \cdots :f_n(v_2))\\
&=& (f_1(v_1) + s_1 \cdot l(e) : \cdots :f_n(v_1) + s_n \cdot l(e)),
\end{eqnarray*}
where each $s_i$ is the slope of $f_i$ on $e$ from $v_1$ to $v_2$.
Let $j$ be a number such that $f_j$ has slope zero on $e$, {\it {\it i.e.}} $s_j = 0$.
Then, the distance between $\phi_{\it F}(v_1)$ and $\phi_{\it F}(v_2)$ is 
\begin{eqnarray*}
&&\text{``the lattice length of }\\
&&((f_1(v_2) - f_j(v_2)) - (f_1(v_1) - f_j(v_1)), \ldots, (f_n(v_2) - f_j(v_2)) - (f_n(v_1) - f_j(v_1)))\text{''}\\
&=& \text{``the lattice length of }(s_1 \cdot l(e), \ldots, s_n \cdot l(e))\text{''}\\
&=& l(e) \cdot {\rm gcd}(s_1, \ldots, s_n) = l(e).
\end{eqnarray*}

Let $(G_{\circ}, l_{\circ})$ (resp. $(G_{\circ}^{\prime}, l_{\circ}^{\prime})$) be the canonical model of $\Gamma$ (resp. ${\rm Im}(\phi_{\it F}))$.
We show that we can choose loopless models $(G_1, l_1)$ and $(G^{\prime}, l^{\prime})$ of $\Gamma$ and ${\rm Im}(\phi_{\it F})$ respectively for that $\phi_{\it F}$ induces a $K$-Galois covering on ${\rm Im}(\phi_{\it F})$ with the edge-multiplicities ${\bold 1} : E(G_1) \rightarrow \boldsymbol{Z}_{\ge 0}, e \mapsto 1$, and $m^{\prime} : E(G^{\prime}) \rightarrow \boldsymbol{Z}_{\ge 0}, e^{\prime} \mapsto |K_{e}|$, where $e$ is an edge of $G_1$ whose image by $\phi_{\it F}$ is $e^{\prime}$. 
For any $x^{\prime} \in V(G_{\circ}^{\prime}) \setminus \phi_{\it F}(V(G_{\circ}))$, since $\phi_{\it F}$ is $K$-injective, there exists a unique orbit $Kx$ in $\Gamma$ whose image by $\phi_{\it F}$ is $x^{\prime}$.
$x$ is smooth.
Let $e$ be the edge of $G_{\circ}$ containing $x$.
If there exist no elements of $K$ which inverse $e$, then $\phi_{\it F}(x)$ is smooth and this is a contradiction.
Hence there exists an element of $K$ which inverse $e$.
As $x^{\prime}$ is not smooth, by the proof of Lemma \ref{$V_1$ is finite}, $x$ is the midpoint of $e$ and $x^{\prime}$ has valence one.
Thus, let $V(G_1) := V(G) \cup \bigcup_{x^{\prime} \in V(G_{\circ}^{\prime}) \setminus V(G_{\circ})} \phi_{\it F}^{-1}(x^{\prime})$ and $V(G^{\prime}) := \phi_{\it F}(V(G_1))$.
Then $\phi_{\it F}$ induces a finite harmonic morphism from $\Gamma$ to ${\rm Im}(\phi_{\it F})$ of degree $|K|$ with the edge-multiplicities ${\bold 1}$ and $m^{\prime}$.
By the definition of the action of $K$ on $\Gamma$ and by the assumption, the induced finite harmonic morphism is a $K$-Galois covering on ${\rm Im}(\phi_{\it F})$.
\end{proof}

\begin{cor}
If $\phi_{\it F}$ is $K$-very ample, then $\phi_{\it F}$ induces a $K$-Galois covering on ${\rm Im}(\phi)$.
\end{cor}

\begin{lemma}
	\label{$K$-Galois is $K$-injective} 
If $\phi_{\it F}$ induces a $K$-Galois covering (with the edge-multiplicities in Theorem \ref{If $K$-injective, then $K$-Galois}), then $\phi_{\it F}$ is $K$-injective.
\end{lemma}

\begin{proof}
If there exist two $K$-orbits in $\Gamma$ whose images by $\phi_{\it F}$ consistent with each other, then the inverse image by $\phi_{\it F}$ contains at least two $K$-orbits.
Thus $K$ does not act transitively on the fibre.
\end{proof}

\begin{cor}
$\phi_{\it F}$ induces a $K$-Galois covering with the edge-multiplicities in Theorem \ref{If $K$-injective, then $K$-Galois} if and only if $\phi_{\it F}$ is $K$-injective.
\end{cor}

\begin{rem}
	\label{$K$-Galois is $K$-injective1}
	\upshape{
By the same proof of Lemma \ref{$K$-Galois is $K$-injective}, we have the statement ``Every $K$-Galois covering on a metric graph (with edge-multiplicities) maps distinct $K$-orbits to distinct points.'', which is more general than Lemma \ref{$K$-Galois is $K$-injective}. 
}
\end{rem}

We then have an answer for the question ``when does $\phi_{\it F}$ induce a $K$-Galois covering on ${\rm Im}(\phi_{\it F})$?''.

Next, we pose a question ``whether there exists a divisor which induces a $K$-Galois covering induced by $K$-invariant linear system or not''.

\begin{rem}[{\cite[Corollary $46$]{Haase=Musiker=Yu}}]
{\upshape
Every divisor of positive degree is ample.
}
\end{rem}

\begin{thm}
	\label{$K$-ample1}
Every effective $K$-invariant divisor of positive degree is $K$-ample.
\end{thm}

\begin{proof}
%\begin{proof}
Let $\pi : \Gamma \rightarrow \Gamma^{\prime} := \Gamma / K$ be the natural surjection.
%Let $\varphi : \Gamma \rightarrow \Gamma^{\prime} := \Gamma / K$ be the natural surjection.
%Let $\varphi : \Gamma \rightarrow \Gamma^{\prime} := \Gamma / K$ be the finite harmonic morphism of degree $|K|$ in Subsection 3.1.
By the construction, $\pi$ is $K$-Galois.
%By the construction, $\varphi$ is $K$-Galois.
Thus $\pi$ is $K$-injective.
%Thus $\varphi$ is $K$-injective.
Let $x$ and $y$ be points on $\Gamma$ whose $K$-orbits are different from each other and let $x^{\prime}:= \pi(x)$ and $y^{\prime}:=\pi(y)$.
%Let $x$ and $y$ be points on $\Gamma$ whose $K$-orbits are different from each other and let $x^{\prime}:= \varphi(x)$ and $y^{\prime}:=\varphi(y)$.
Let $D$ be an effective $K$-invariant divisor on $\Gamma$ of positive degree.
$\pi_{\ast}(D)$ is ample since ${\rm deg}(\pi_{\ast}(D))={\rm deg}(D) \ge 1$.
%$\varphi_{\ast}(D)$ is ample since ${\rm deg}(\varphi_{\ast}(D))={\rm deg}(D) \ge 1$.
Therefore there exists a positive integer $k$ such that $k\pi_{\ast}(D)$ is very ample.
%Therefore there exists a positive integer $k$ such that $k\varphi_{\ast}(D)$ is very ample.
Let $f_1^{\prime}$ and $f_2^{\prime}$ be in $R(k\pi_{\ast}(D))$ such that $f_1^{\prime}(x^{\prime})-f_1^{\prime}(y^{\prime}) \not= f_2^{\prime}(x^{\prime})-f_2^{\prime}(y^{\prime})$.
%Let $f_1^{\prime}$ and $f_2^{\prime}$ be in $R(k\varphi_{\ast}(D))$ such that $f_1^{\prime}(x^{\prime})-f_1^{\prime}(y^{\prime}) \not= f_2^{\prime}(x^{\prime})-f_2^{\prime}(y^{\prime})$.
As $D$ is $K$-invariant and $\pi$ is $K$-injective,
%As $D$ is $K$-invariant and $\varphi$ is $K$-injective,
\begin{eqnarray*}
\pi^{\ast} \left( \pi_{\ast}(D) \right) &=& \pi^{\ast}\left( \sum_{x \in \Gamma}D(x)\cdot \pi(x) \right)\\
%\varphi^{\ast} \left( \varphi_{\ast}(D) \right) &=& \varphi^{\ast}(\sum_{x \in \Gamma}D(x)\cdot \varphi(x))\\
&=& \sum_{x \in \Gamma}{\rm deg}_x(\pi) \cdot \left\{ \left( \sum_{y \in \Gamma}D(y)\cdot \pi(y) \right) (\pi(x)) \right\} \cdot x\\
%&=& \sum_{x \in \Gamma}{\rm deg}_x(\varphi) \cdot \left\{ \left( \sum_{y \in \Gamma}D(y)\cdot \varphi(y) \right) (\varphi(x)) \right\} \cdot x\\
&=& \sum_{x \in \Gamma}{\rm deg}_x(\pi) \cdot \left( \sum_{y \in \pi^{-1}(\pi(x))}D(y) \right) \cdot x\\
%&=& \sum_{x \in \Gamma}{\rm deg}_x(\varphi) \cdot \left( \sum_{y \in \varphi^{-1}(\varphi(x))}D(y) \right) \cdot x\\
&=& \sum_{x \in \Gamma}{\rm deg}_x(\pi) ( |Kx| \cdot D(x) ) \cdot x = \sum_{x \in \Gamma}(|K_x| \cdot |Kx| \cdot D(x))\cdot  x\\
%&=& \sum_{x \in \Gamma}{\rm deg}_x(\varphi) ( |Kx| \cdot D(x) ) \cdot x = \sum_{x \in \Gamma}(|K_x| \cdot |Kx| \cdot D(x))\cdot  x\\
&=& \sum_{x \in \Gamma} |K| D(x) \cdot x = |K| D.
\end{eqnarray*}
Since $k\pi_{\ast}(D) + {\rm div}(f_i^{\prime})$ is effective,
%Since $k\varphi_{\ast}(D) + {\rm div}(f_i^{\prime})$ is effective,
\[
\pi^{\ast}(k\pi_{\ast}(D) + {\rm div}(f_i^{\prime})) = k\pi^{\ast}(\pi_{\ast}(D))+\pi^{\ast}({\rm div}(f_i^{\prime})) = k|K|D+{\rm div}(\pi^{\ast}f_i^{\prime})
%\varphi^{\ast}(k\varphi_{\ast}(D) + {\rm div}(f_i^{\prime})) = k\varphi^{\ast}(\varphi_{\ast}(D))+\varphi^{\ast}({\rm div}(f_i^{\prime})) = k|K|D+{\rm div}(\varphi^{\ast}f_i^{\prime})
\]
is also effective.
This means $\pi^{\ast}f_i^{\prime} \in R(k|K|D)$.
%This means $\varphi^{\ast}f_i^{\prime} \in R(k|K|D)$.
As
\begin{eqnarray*}
\pi^{\ast}f_1^{\prime}(x)-\pi^{\ast}f_1^{\prime}(y) &=& f_1^{\prime}(\pi(x))-f_1^{\prime}(\pi(y)) = f_1^{\prime}(x^{\prime})-f_1^{\prime}(y^{\prime})\\
%\varphi^{\ast}f_1^{\prime}(x)-\varphi^{\ast}f_1^{\prime}(y) &=& f_1^{\prime}(\varphi(x))-f_1^{\prime}(\varphi(y)) = f_1^{\prime}(x^{\prime})-f_1^{\prime}(y^{\prime})\\
&\not=& f_2^{\prime}(x^{\prime})-f_2^{\prime}(y^{\prime}) = f_2^{\prime}(\pi(x))-f_2^{\prime}(\pi(y))\\
%&\not=& f_2^{\prime}(x^{\prime})-f_2^{\prime}(y^{\prime}) = f_2^{\prime}(\varphi(x))-f_2^{\prime}(\varphi(y))\\
&=& \pi^{\ast}f_2^{\prime}(x)-\pi^{\ast}f_2^{\prime}(y),\end{eqnarray*}
%&=& \varphi^{\ast}f_2^{\prime}(x)-\varphi^{\ast}f_2^{\prime}(y),\end{eqnarray*}
$k|K|D$ is $K$-very ample.
\end{proof}

Therefore, the answer is ``always''.

In conclusion, we have the following theorem.

\begin{thm}
	\label{main theorem2}
Let $\Gamma$ be a metric graph and $K$ a finite group acting on $\Gamma$.
Then, there exists a rational map, from $\Gamma$ to a tropical projective space, which induces a $K$-Galois covering on the image with edge-multiplicities.
\end{thm}

Especially when the group $K$ is trivial, we have the following corollary.

\begin{cor}
	\label{embedded in a tropical projective space1}
A metric graph is embedded in a tropical projective space by a rational map.
\end{cor}

\begin{prop}
	\label{pull-back of rational functions2}
If $\phi_{\it F}$ induces a $K$-Galois covering $\phi$, then $\phi^{\ast}({\rm Rat}({\rm Im}(\phi_{\it F})) = {\rm Rat}(\Gamma)^K$ holds.
%%If $\phi_{|D|^K}$ induces a $K$-Galois covering $\phi$, then $\phi^{\ast}({\rm Rat}({\rm Im}(\phi_{|D|^K})) = {\rm Rat}(\Gamma)^K$ holds.
\end{prop}

\begin{proof}
For  any $f^{\prime} \in {\rm Rat}({\rm Im}(\phi_{\it F}))$, obviously $\phi^{\ast}(f^{\prime}) = f^{\prime} \circ \phi \in {\rm Rat}(\Gamma)^K$ holds.
%%For  any $f^{\prime} \in {\rm Rat}({\rm Im}(\phi_{|D|^K})$, Obviously $\phi^{\ast}(f^{\prime}) = f^{\prime} \circ \phi \in {\rm Rat}(\Gamma)^K$ holds.

Let $(G, l)$ (resp. $(G^{\prime}, l^{\prime}))$ be a model of $\Gamma$ (resp. ${\rm Im}(\phi_{\it F})$) corresponding to $\phi$.
%%Let $(G, l)$ (resp. $(G^{\prime}, l^{\prime}))$ be a model of $\Gamma$ (resp. ${\rm Im}(\phi_{|D|^K})$) corresponding to $\phi$.
Let $f$ be an element of ${\rm Rat}(\Gamma)^K$.
Since $\varphi$ is $K$-injective, there exists a one-to-one mapping between $K$-orbits of $\Gamma$ and ${\rm Im}(\phi_{\it F})$.
%%Since $\varphi$ is $K$-injective, there exists a one-to-one mapping between $K$-orbits of $\Gamma$ and ${\rm Im}(\phi_{|D|^K})$.
Let $g(x^{\prime}) := f(\phi^{-1}(x^{\prime})), x^{\prime} \in {\rm Im}(\phi_{\it F})$ and $g$ is well-defined.
%%Let $g(x^{\prime}) := f(\phi^{-1}(x^{\prime})), x^{\prime} \in {\rm Im}(\phi_{|D|^K})$ and $g$ is well-defined.
By the definition of $g$, for any $x \in \Gamma$, $\phi^{\ast}(g) (x) = g \circ \phi (x) = g(\phi(x)) = f(x)$ holds.
Thus, $f = \phi^{\ast}(g) \in \phi^{\ast}({\rm Rat}({\rm Im}(\phi_{\it F})))$.
%%Thus, $f = \phi^{\ast}(g) \in \phi^{\ast}({\rm Rat}({\rm Im}(\phi_{|D|^K})))$.
%%As $\phi$ is an isometry, for any connected component $e$ of $\Gamma \setminus ({\rm supp}(f) \cup V(G))$, the slope of $f$ on $e$ is equal to the one of $g$ on $\phi(e)$.
\end{proof}

\begin{rem}
	\label{pull-back of rational functions}
Let $\Gamma$ be a metric graph, $K$ a finite group acting on $\Gamma$ and $\varphi : \Gamma \rightarrow \Gamma^{\prime} := \Gamma /K$ the natural surjection.
Let $(G_1, l_1)$ (resp. $(G^{\prime}, l^{\prime})$) be the model of $\Gamma$ (resp. $\Gamma^{\prime}$) in Section $3$.
${\rm Rat}(\Gamma)_K$ denotes the set consisting of $K$-invariant rational functions $f$ on $Gamma$ whose each slope on $e$ is a multiple of $|K_e|$, where $e$ is a connected component of $\Gamma \setminus ({\rm supp}({\rm div}(f)) \cup V(G_1))$.
Then, $\varphi^{\ast}({\rm Rat}(\Gamma^{\prime})) = {\rm Rat}(\Gamma)_K$.
\end{rem}

\begin{proof}
Let $f^{\prime} \in {\rm Rat}(\Gamma^{\prime})$.
By the definition of pull-back of a function, $\varphi^{\ast}(f^{\prime}) = f^{\prime} \circ \varphi \in {\rm Rat}(\Gamma)^K$.
Let $e$ be a connected component of $\Gamma \setminus ({\rm supp}({\rm div}(\varphi^{\ast}(f^{\prime}))) \cup V(G_1))$.
By the construction of $\Gamma^{\prime}$, $l^{\prime}(\varphi(e)) = l^{\prime}([e]) = |K_e| l(e)$.
$\varphi^{\ast}(f^{\prime})$ has the slope which is a multiple by $K_e$ of the one of $f^{\prime}$ on $\varphi(e)$.
Therefore, $\varphi^{\ast}(f^{\prime})$ is in ${\rm Rat}(\Gamma)_K$.

Let $f$ be an element of ${\rm Rat}(\Gamma)_K$.
Let $g$ be the rational function on $\Gamma^{\prime}$ defined by the following $(1)$ and $(2)$.
\begin{itemize}
\item[$(1)$] Fix a point $x_0$ on $\Gamma$.
$g(\varphi(x_0)) := f(x_0)$.
\item[$(2)$] For a connected component $e$ of $\Gamma \setminus ({\rm supp}({\rm div}(f)) \cup V(G_1))$, $g$ has the slope $\frac{(\text{the slope of } f \text{ on } e)}{|K_e|}$.
\end{itemize}
Then, $g$ is well-defined and $f = \varphi^{\ast}(g) \in \varphi^{\ast}({\rm Rat}(\Gamma^{\prime})$.
In fact, the following fold.
Let $x_1, x_2$ be any two point on $\Gamma$ and $P_1 = e_{11} \cdots e_{1n_1}$ and $P_2 = e_{21} \cdots e_{2n_2}$ any two paths from $x_1$ to $x_2$.
Let $s_{ij}$ be the slope of $f$ on $e_{ij}$.
As
\[
f(x_2) = f(x_1) + \sum_{j=1}^{n_1} l_1(e_{1j})s_{1j} = f(x_1) + \sum_{j=1}^{n_2} l_1(e_{2j})s_{2j}
\]
holds, then we have
\[
l_1(e_{1j})s_{1j} = \sum_{j=1}^{n_2} l_1(e_{2j})s_{2j}.
\]
Therefore,
\[
\sum_{j=1}^{n_1} \frac{ |K_{e_{1j}}| \cdot l_1(e_{1j}) \cdot s_{1j}}{|K_{e_{1j}}|} = \sum_{j=1}^{n_1} \frac{ |K_{e_{2j}}| \cdot l_1(e_{1j}) \cdot s_{2j}}{|K_{e_{2j}}|}
\]
and then $g$ is well-defined.
Let $x_1$ be $x_0$.
For any $x_2$,
\begin{eqnarray*}
f(x_2) &=& f(x_0) + \sum_{j=1}^{n_1} l_1 (e_{1j}) s_{1j}\\
&=& g(\varphi(x_0)) + \sum_{j=1}^{n_1} \frac{ |K_{e_{ij}}| \cdot l_1 (e_{1j}) \cdot s_{1j}}{|K_{e_{ij}}|} = g(\varphi(x_2)).
\end{eqnarray*}
Then, $f = \varphi^{\ast}(g)$.
\end{proof}

\subsection{Applications}

In \cite{Haase=Musiker=Yu}, Haase, Musiker and Yu give a problem ``give a characterization of metric graphs whose canonical divisors are not very ample'' (see {\cite[Problem $51$]{Haase=Musiker=Yu}}).
In this subsection, we give an answer to this problem and at the same time, we consider an analogy of the fact the canonical map of a hyperelliptic compact Riemann surface is a double covering.

Let $\Gamma$ be a metric graph and $D$ a divisor on $\Gamma$.
$\phi_{|D|}$ denotes the rational map induced by $|D|$, {\it {\it i.e.}} for a minimal generating set $\{ f_1, \ldots, f_n \}$ of $R(D)$, $\phi_{|D|} := (f_1 : \cdots : f_n) : \Gamma \rightarrow \boldsymbol{TP}^{n-1}, x \mapsto (f_1(x) : \cdots : f_n(x))$.

\begin{rem}[{\cite[Proposition $48$]{Haase=Musiker=Yu}}]
	\label{hyperelliptic proposition}
\upshape{
If ${\rm deg}(D)=2$, then $\phi_{|D|}(\Gamma)$ is a tree.
If in addition $r(D)=1$, then the fibre $\phi_{|D|}^{-1}(x)=\{y \in \Gamma \,|\, \phi_{|D|}(y)=x \}$ has size one or two for all $x$ in the image.
}
\end{rem}

By Remark \ref{hyperelliptic proposition}, we have the following lemma.

\begin{lemma}
	\label{hyperelliptic lemma}
Let $\Gamma$ be a hyperelliptic metric graph without one valent points and $D$ a divisor on $\Gamma$ whose degree is two and whose rank is one.
Then, the complete linear system $|D|$ is invariant by the hyperelliptic involution $\iota$ and the rational map associated to $|D|$ induces a $\langle \iota \rangle$-Galois covering on a tree.
%Then, the complete linear system $|D|$ is invariant by the hyperelliptic involution $\iota$ and the rational map associated to $|D|$ induces a branched $\langle \iota \rangle$-Galois covering on a tree.
\end{lemma}

\begin{proof}
Obviously $|D|$ is invariant by $\langle \iota \rangle$.
By Remark \ref{hyperelliptic proposition}, ${\rm Im}(\phi_{|D|})$ is a tree.
By the proof of Remark \ref{hyperelliptic proposition}, for any point $x$ on a bridge of $\Gamma$, $|\phi_{|D|}^{-1}(\phi_{|D|}(x))| = 1$ and any point $y$ not on a bridge but on a cycle of $\Gamma$,$|\phi_{|D|}^{-1}(\phi_{|D|}(y))| = 2$ and $\phi_{|D|}^{-1}(\phi_{|D|}(y)) = \{y, \iota(y) \}$.
Therefore $\phi_{|D|}$ is $\langle \iota \rangle$-injecive.
Thus $\phi_{|D|}$ induces a $\langle \iota \rangle$-Galois covering.
%Thus $\phi_{|D|}$ induces a branched $\langle \iota \rangle$-Galois covering.
\end{proof}

The {\it canonical map} is the rational map induced by the canonical linear system $|K_{\Gamma}|$ on a metric graph $\Gamma$.

\begin{thm}
	\label{canonical map1}
Let $\Gamma$ be a metric graph without one valent points and $\phi_{|K_{\Gamma}|}$ the canonical map of $\Gamma$.
Then $\phi_{|K_{\Gamma}|}$ induces a $\boldsymbol{Z} / 2 \boldsymbol{Z}$-Galois covering on the image of $\phi_{|K_{\Gamma}|}$ if and only if the genus $g$ of $\Gamma$ is two.
%Then $\phi_{|K_{\Gamma}|}$ induces a branched $\boldsymbol{Z} / 2 \boldsymbol{Z}$-Galois covering on the image of $\phi_{|K_{\Gamma}|}$ if and only if the genus $g$ of $\Gamma$ is two.
\end{thm}

\begin{proof}
Since ${\rm deg}(K_{\Gamma}) = 2g - 2$ and $r(K_{\Gamma}) = g-1$ by Riemann--Roch theorem, when $g = 0$, $\phi_{|K_{\Gamma}|}$ is not induced and when $g = 1$, $\phi_{|K_{\Gamma}|}$ is a constant map.
When $g = 2$, $K_{\Gamma}$ has degree two and rank one and then by Lemma \ref{hyperelliptic lemma}, $\phi_{|K_{\Gamma}|}$ is a $\boldsymbol{Z} / 2 \boldsymbol{Z}$-Galois covering on a tree.
%When $g = 2$, $K_{\Gamma}$ has degree two and rank one and then by Lemma \ref{hyperelliptic lemma}, $\phi_{|K_{\Gamma}|}$ is a branched $\boldsymbol{Z} / 2 \boldsymbol{Z}$-Galois covering on a tree.
When $g \ge 3$, for $K_{\Gamma}$ is not very ample, $\Gamma$ must be one of the following two type of hyperelliptic metric graphs by \cite[Theorem 49]{Haase=Musiker=Yu}.

(type $1$)
$\Gamma$ is a metric graph consisting two vertices $x, y$ and $g + 1$ multiple edges between them.
See Figure \ref{canonicalmap-1}.

\begin{figure}[htbp]
 \begin{center}
  \includegraphics[width=80mm]{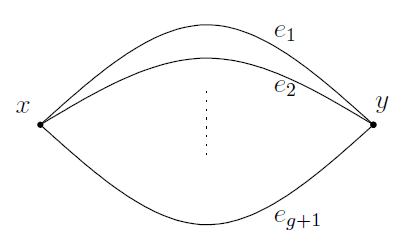}
 \end{center}
 \caption{type $1$}
 \label{canonicalmap-1}
\end{figure}

The rational functions $f_{i1}, f_{i2}$ and $f_{i3}$ in Figure \ref{canonicalmap-3} are extremals of $R(K_{\Gamma})$ and the rational map $e_i^{\circ} \rightarrow \boldsymbol{TP}^{2}, t \mapsto (f_{i1}(t) : f_{i2}(t) : f_{i3}(t))$ is injective.

\begin{figure}[h]
 \begin{center}
  \includegraphics[width=160mm]{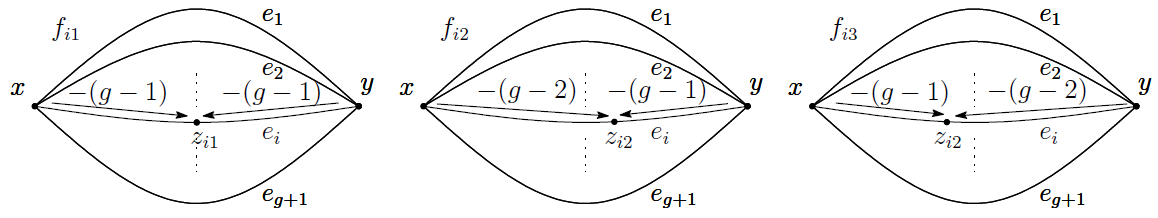}
 \end{center}
 \caption{$z_{i1}$ is the midpoint of $e_i$.
 $z_{i2}$ and $z_{i3}$ are the internally dividing points obtained by internally dividing $e_i$ by $(g-1) : (g-2)$, where $z_{i2}$ is further than $z_{i3}$ from $x$.
 $f_{i1}, f_{i2}$ and $f_{i3}$ define principal divisors such that $D + {\rm div}(f_{i1}) = (2g - 2)z_{i1}$, $D + {\rm div}(f_{i2}) = x + (2g - 3)z_{i2}$ and $D + {\rm div}(f_{i2}) = y + (2g - 3)z_{i3}$, respectively.}
 \label{canonicalmap-3}
\end{figure}

On the other hand, obviously all extremals of $R(K_{\Gamma})$ attain maximal only at $x$ and $y$ by Lemma \ref{condition to be extremal} (in this case, $K$ is trivial).
Hence $\phi_{|K_{\Gamma}|}|_{\Gamma \setminus \{ x, y \}}$ is injective and $\phi_{|K_{\Gamma}|}(x) = \phi_{|K_{\Gamma}|}(y)$.
Thus $\phi_{|K_{\Gamma}|}$ is not a $\boldsymbol{Z} / 2 \boldsymbol{Z}$-Galois covering.
%Thus $\phi_{|K_{\Gamma}|}$ is not a branched $\boldsymbol{Z} / 2 \boldsymbol{Z}$-Galois covering.

(type $2$)
$\Gamma$ is a metric graph of the form in Figure \ref{canonicalmap-2}.
$e_{g + 2}$ and $e_{g + 3}$ have a same length.
\begin{figure}[htbp]
 \begin{center}
  \includegraphics[width=80mm]{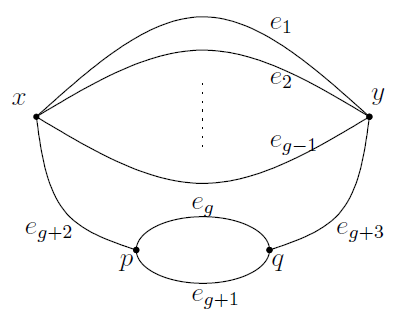}
 \end{center}
 \caption{type $2$}
 \label{canonicalmap-2}
\end{figure}

Since $K_{\Gamma}$ is linearly equivalent to $D := (g-1)(x+y)$, $\phi_{|K_{\Gamma}|} = \phi_{|D|}$ holds.
Similarly to the proof of type $1$, we have three extremals $f_{i1}, f_{i2}$ and $f_{i3}$ of $R(D)$ which induce an injective rational map on $e_i^{\circ}, i=1, \ldots, g$.
The rational functions $h_1, \ldots, h_5$ and $h_6$ in Figure \ref{canonicalmap-4} are extremals of $R(D)$ and the map $(e_{g} \cup e_{g+1}) \setminus \{ p,q \} \rightarrow \boldsymbol{TP}^{5}, t \mapsto (h_1(t) : h_2(t) : h_3(t) : h_4(t) : h_5(t) : h_6(t))$ is injective.
The rational functions $h_7$ and $h_8$ in Figure \ref{canonicalmap-5} are extremals of $R(D)$ and the map $(e_{g + 2} \cup e_{g + 3}) \setminus \{ x,y \} \rightarrow \boldsymbol{TP}^{2}, t \mapsto (f_{11}(t) : h_7(t) : h_8(t))$ is injective.
In particular, when $g = 3$, see Figure \ref{canonicalmap-7}.
Hence $\phi_{|D|}|_{\Gamma \setminus \{ x, y \}}$ is injective.
On the other hand, by the same reason, $\phi_{|D|}(x) = \phi_{|D|}(y)$.
In conclusion, $\phi_{|D|}$ is not a $\boldsymbol{Z} / 2 \boldsymbol{Z}$-Galois covering.
%In conclusion, $\phi_{|D|}$ is not a branched $\boldsymbol{Z} / 2 \boldsymbol{Z}$-Galois covering.

\begin{figure}[htbp]
 \begin{center}
  \includegraphics[width=160mm]{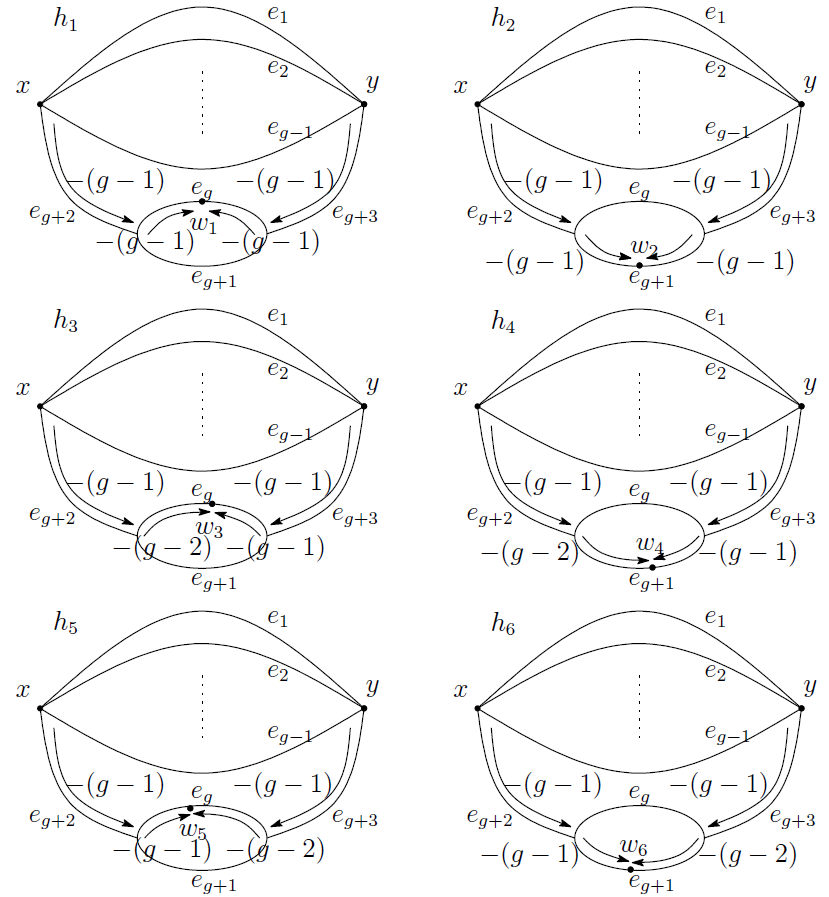}
 \end{center}
 \caption{$w_1$(resp.~$w_2$) is the midpoint of $e_g$(resp.~$e_{g+1}$).
$w_3$ and $w_5$(resp.~$w_4$ and $w_6$) are the internally dividing points obtained by internally dividing $e_g$(resp.~$e_{g+1}$) by $(g-2) : (g-1)$, where $w_3$(resp.~$w_4$) is further than $w_5$(resp.~$w_6$) from $p$(resp.~$q$).
$h_{1}, h_{2}, h_{3}, h_{4}, h_{5}$ and $h_{6}$ define principal divisors such that $D + {\rm div}(h_1) = (2g-2)w_1$, $D + {\rm div}(h_2) = (2g-2)w_2$, $D + {\rm div}(h_3) = p + (2g-3)w_3$, $D + {\rm div}(h_4) = p + (2g-3)w_4$, $D + {\rm div}(h_5) = q + (2g-3)w_5$ and $D + {\rm div}(h_6) = q + (2g-3)w_6$, respectively.}
 \label{canonicalmap-4}
\end{figure}

\begin{figure}[htbp]
 \begin{center}
  \includegraphics[width=160mm]{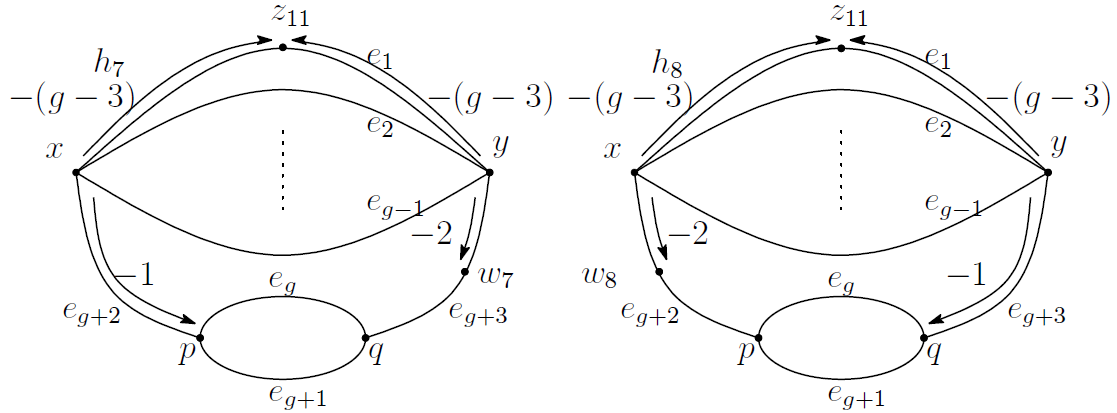}
 \end{center}
 \caption{(case $g \ge 4$) $z_{11}, w_7$ and $w_8$ are the midpoints of $e_1, e_{g + 2}$ and $e_{g + 3}$, respectively.
$h_{7}$ and $h_{8}$ define principal divisors such that $D + {\rm div}(h_7) = (2g-6)z_{11} + p + 2w_7$ and $D + {\rm div}(h_8) = (2g-6)z_{11} + q + 2w_8$, respectively.}
 \label{canonicalmap-5}
\end{figure}

\begin{figure}[htbp]
 \begin{center}
  \includegraphics[width=160mm]{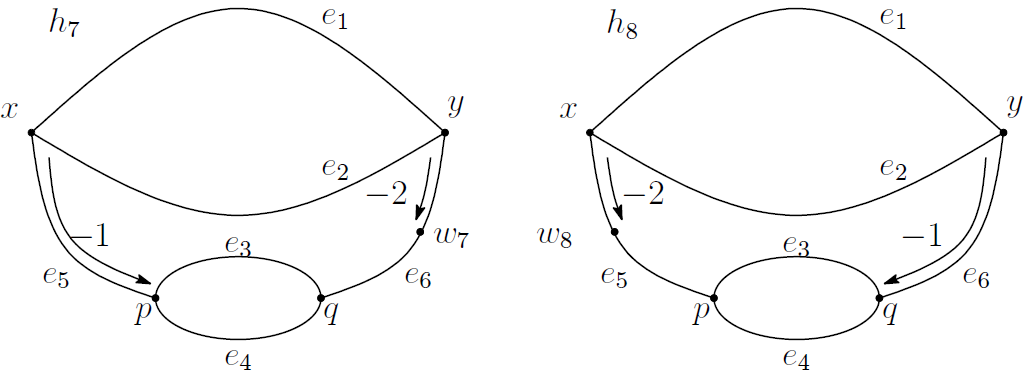}
 \end{center}
 \caption{(case $g = 3$) $w_7$ and $w_8$ are the midpoints of $e_{g + 2}$ and $e_{g + 3}$, respectively.
 $h_{7}$ and $h_{8}$ define principal divisors such that $D + {\rm div}(h_7) = p + 2w_7$ and $D + {\rm div}(h_8) = q + 2w_8$, respectively.}
 \label{canonicalmap-7}
\end{figure}

\end{proof}

\begin{cor}
	\label{canonical map2}
Let $\Gamma$ be a metric graph of genus $\ge 3$ without one valent points.
$K_{\Gamma}$ is not very ample if and only if the canonical map is not harmonic.
In particular, $\Gamma$ is hyperelliptic and $g({\rm Im}(\phi_{|K_{\Gamma}|})=g(\Gamma)+1$.
\end{cor}

\begin{proof}
In the proof of Theorem \ref{canonical map1}, we can directly check that the degrees of $\phi_{|K_{\Gamma}|}$ (resp. $\phi_{|D|}$) at $x$ and $y$ are different from the degree of $\phi_{|K_{\Gamma}|}$ (resp. $\phi_{|D|}$) at any other point.
\end{proof}

%\begin{rem}
%\upshape{
%We also have Corollary \ref{canonical map2} by the fact ``for metric graphs $\Gamma$ and $\Gamma^{\prime}$, if there exists a harmonic morphism $\varphi : \Gamma \rightarrow \Gamma^{\prime}$, then $g(\Gamma) \ge g(\Gamma^{\prime})$''.
%A proof of the fact is put in Appendix.
%See Proposition \ref{genus proposition} in Appendix.
%}
%\end{rem}

Theorem \ref{canonical map1} and Corollary \ref{canonical map2} mean that an analogy of the fact the canonical map of a hyperelliptic compact Riemann surface is a double covering on a projective line $\boldsymbol{P}^1(\boldsymbol{C})$ with non-zero degree does not hold for a metric graph and in stead of this, we have the following by Lemma \ref{hyperelliptic lemma}.

\begin{prop}
	\label{double covering1}
Let $\Gamma$ be a hyperelliptic metric graph with genus at least two without one valent points.
Then, an invariant linear subsystem of the hyperelliptic involution $\iota$ of the canonical linear system induces a rational map whose image is a tree and which is a $\langle \iota \rangle$-Galois covering on the image.
%Then, an invariant linear subsystem of the hyperelliptic involution $\iota$ of the canonical linear system induces a rational map whose image is a tree and which is a branched $\langle \iota \rangle$-Galois covering on the image.
\end{prop}

\begin{proof}
As $\iota$ is an isometry, $K_{\Gamma}$ is of the form $D + E$, where both $D$ and $E$ are effective divisors on $\Gamma$ and ${\rm deg}(D) = 2$, $r(D)=1$.
Since $K_{\Gamma}$ and $D$ are invariant by $\iota$, so is $E$.
Thus the canonical linear system $|K_{\Gamma}| = |D + E|$ contains the invariant linear subsystem $\Lambda$ of the hyperelliptic involution whose elements are of the form $D_1 + E$, where $D_1$ is effective and linearly equivalent to $D$.
Let $R$ be the subsemimodule of $R(K_{\Gamma}) = R(D + E)$ corresponding to $\Lambda$.
Then $R = R(D)$.
In fact, for any $f \in R$, there exists $D_1 + E \in \Lambda$ such that $(D + E) + {\rm div}(f) = D_1 + E \ge 0$ and $D_1$ is  effective.
Thus $D + {\rm div}(f) = D_1 \ge 0$, {\it i.e.} $f \in R(D)$.
Conversely, for any $g \in R(D)$, there exists $D_1 \in \Lambda^{\prime}$ such that $D + {\rm div}(g) = D_1 \ge 0$, where $\Lambda^{\prime}$ is the linear system corresponding to $R(D)$.
Hence $D + E + {\rm div}(g) = D_1 + E \ge 0$ and then $g \in R$.
Therefore, by Lemma \ref{hyperelliptic lemma}, $\Lambda$ induces a rational map which is a $\langle \iota \rangle$-Galois covering on a tree.
%Therefore, by Lemma \ref{hyperelliptic lemma}, $\Lambda$ induces a rational map which is a branched $\langle \iota \rangle$-Galois covering on a tree.
\end{proof}

Moreover, we have the following lemma.

\begin{lemma}
	\label{containment relation lemma}
Let $\Gamma$ be a metric graph, $K$ a finite group and $D$ a divisor on $\Gamma$.
For finitely generated $K$-invariant linear subsystems $\Lambda_1 \subset \Lambda_2 \subset |D|$, let $\phi_{\Lambda_1} = (f_1 : \cdots : f_n)$ (resp. $\phi_{\Lambda_2} = (g_1 : \cdots : g_m)$) be the rational map induced by $\Lambda_1$ (resp. $\Lambda_2$).
If $\phi_{\Lambda_1}$ induces a $K$-Galois covering, then $\phi_{\Lambda_1}$ induces a $K$-Galois covering.
\end{lemma}

\begin{proof}
Let $f_i = \text{``}\sum_{j=1}^{m} a_{ij}g_j\text{''}$.
Since $m$ is finite and each $g_i$ is a rational function on a metric graph, we can choose a model $(G, l)$ of $\Gamma$ satisfying the following condition: for any $i$ and edge $e$ of $G$, there exists a number $i_e$ such that $f_i |_e = a_{ii_e} + g_{i_e}$.
Let $f^{\prime}_i := f_i - a_{ii_e}$.
Then, ${\it F}^{\prime} := \{ f^{\prime}_1, \ldots, f^{\prime}_n \}$ is a minimal generating set of $R_1$, where $R_1$ is the tropical subsemimodule of $R(D)$ corresponding to $\Lambda_1$.
As $\phi_{\Lambda_1}$ is $K$-Galois, by Theorem \ref{If $K$-injective, then $K$-Galois}, it is $K$-injective and then $\phi_{{\it F}^{\prime}}$ is also $K$-injective by Remark \ref{independent remark}.
By the definition of ${\it F}^{\prime}$, $\phi_{\Lambda_2}|e = \phi_{{\it F}^{\prime}}|e$ holds.
Hence $\phi_{\Lambda_2}$ is $K$-injective on $e$.
Since $e$ is arbitrary, $\phi_{\Lambda_2}$ is $K$-injective.
In conclusion, by Theorem \ref{If $K$-injective, then $K$-Galois} again, $\phi_{\Lambda_2}$ induces a $K$-injective.
%In conclusion, by Theorem \ref{If $K$-injective, then $K$-Galois} again, $\phi_{\Lambda_2}$ induces a branched $K$-injective.
\end{proof}

Consequently, by Proposition \ref{double covering1} and Lemma \ref{containment relation lemma}, the following holds.

\begin{thm}
	\label{canonical map3}
For a hyperelliptic metric graph with genus at least two without one valent points, the invariant linear system of the hyperelliptic involution $\iota$ of the canonical linear system induces a rational map whose image is a tree and which is a $\langle \iota \rangle$-Galois covering on the image.
%For a hyperelliptic metric graph with genus at least two without one valent points, the invariant linear system of the hyperelliptic involution $\iota$ of the canonical linear system induces a rational map whose image is a tree and which is a branched $\langle \iota \rangle$-Galois covering on the image.
\end{thm}

\section{Metric graphs with edge-multiplicities}

In Section $3$, we prove that the induced rational map by $|D|^K$ which is $K$-injective is a finite harmonic morphism (and then a $K$-Galois covering) of metric graphs with an edge-multiplicity.
%In Section $5$, we prove that the induced rational map by $|D|^K$ which is $K$-injective is a finite harmonic morphism (and then a $K$-Galois covering) of metric graphs with an edge-multiplicity.
We define in this section metric graphs with edge-multiplicities and harmonic morphisms between them.
%We define in this section a metric graph with an edge-multiplicity and several concepts corresponding to a metric graph with an edge-multiplicity.
%Compare Subsections $2.2$ and $2.4$.
Compare Subsections $2.2, 2.3$ and $2.4$.
Note that all of them are original definitions of the author and we may need more improvements.

\subsection{Metric graphs with edge-multiplicities}

\begin{dfn}
{\upshape
Let $\Gamma$ be a metric graph, and $(G, l)$ a model of $\Gamma$.
We call a function $m : E(G) \rightarrow \Z_{>0}$ an {\it edge-multiplicity function} on $G$.
${\bold 1}$ is the edge-multiplicity function assigning multiplicity one to all edges and called a {\it trivial} edge-multiplicity function.
%%We call a function $m_G : E(G) \rightarrow \Z_{>0}$ an {\it edge-multiplicity function} on $G$.
%For an edge $e$, $m(e)$ is said to be {\it multiplicity} of $e$, and in this case, we say that $e$ {\it has} multiplicity $m(e)$.
%%For an edge $e$, $m_G(e)$ is said to be {\it multiplicity} of $e$, and in this case, we say that $e$ {\it has} multiplicity $m_G(e)$.
Two triplets $(G, l, m)$ and $(G^{\prime}, l^{\prime}, m^{\prime})$ are said to be {\it isomorphic} if there exists an isomorphism between $G$ and $G^{\prime}$ keeping the length and the multiplicity of each edge.
%%Two triplets $(G, l, m_G)$ and $(G^{\prime}, l^{\prime}, m_G^{\prime})$ are said to be {\it isomorphic} if there exists an isomorphism between $G$ and $G^{\prime}$ keeping the length and the multiplicity of each edge.
We define ${\rm Isom}_{(G, l)}(\Gamma)$ as the subset of the isometry transformation group ${\rm Isom}(\Gamma)$ of $\Gamma$ whose element keeps the length of each edge of $G$.
%We define ${\rm Aut}_{(G, l)}(\Gamma)$ as the set of automorphisms of $\Gamma$ which keep the length of each edge of $G$.
%%We define ${\rm Aut}_{(G, l)}(\Gamma)$ as the set of automorphism of $\Gamma$ induced from some automorphism of $G$ keeping the length of each edge.
We set ${\rm Isom}_{(G, l, m)}$ as the subset of ${\rm Isom}_{(G, l)}(\Gamma)$ whose each element keeps the multiplicity of each edge of $G$.
%We set ${\rm Aut}_{(G, l, m)}$ as the subset of ${\rm Aut}_{(G, l)}(\Gamma)$ whose each element keeps the length and the multiplicity of each edge of $G$.
%%We set ${\rm Aut}_{(G, l, m_G)}$ the subset of ${\rm Aut}_{(G, l)}(\Gamma)$ whose each element is induced from some automorphism of $G$ keeping the length and the multiplicity of each edge.
}
\end{dfn}

\begin{dfn}[Subdivision of models]
{\upshape
Let $\Gamma$ be a metric graph, and $(G, l)$, $(G^{\prime}, l^{\prime})$ models of $\Gamma$.
$(G, l)$ is said to be a {\it subdivision} of $(G^{\prime}, l^{\prime})$ and written as $(G, l) \succ (G^{\prime}, l^{\prime})$ if $V(G^{\prime})$ is a subset of $V(G)$.
%$(G, l)$ is said to be a {\it subdivision} of $(G^{\prime}, l^{\prime})$ and written as $(G, l) \succ (G^{\prime}, l^{\prime})$ if these satisfy the following two conditions:
%$(1)$ $V(G^{\prime})$ is a subset of $V(G)$, and
%$(2)$ for any edge $e^{\prime}$ of $G^{\prime}$, there exist distinct edges $e_1, \ldots, e_n$ of $G$ such that $e^{\prime} = e_1 \sqcup \cdots \sqcup e_n$.
%$(2)$ for any edge $e^{\prime}$ of $G^{\prime}$ and distinct edges $e_1, \ldots, e_n$ of $G$ such that $e^{\prime} = e_1 \cup \cdots \cup e_n$, the length of $e^{\prime}$ by $l^{\prime}$ is the sum of the lengths of $e_1, \cdots, e_n$ by $l$.
%%$(2)$ for any $e^{\prime} = e_1 \cup \cdots \cup e_n \in E(G^{\prime})$, where each $e_i$ is in $E(G)$ and $e_1, \cdots, e_n$ are different from each other, the length of $e^{\prime}$ by $l^{\prime}$ is the sum of the lengths of $e_1, \cdots, e_n$ by $l$.
}
\end{dfn}

%\begin{rem}
%For a model $(G^{\prime}, l^{\prime})$ of a metric graph $\Gamma$ and a graph $G$ of $\Gamma$ whose vertices include the vertices of $G^{\prime}$, there exists unique length function $l: E(G) \rightarrow \boldsymbol{R}_{>0}$ such that $(G, l)$ is a subdivision of $(G^{\prime}, l^{\prime})$.
%\end{rem}
\begin{dfn}
{\upshape
Let $\Gamma$ be a metric graph, and $(G, l) \succ (G^{\prime}, l^{\prime})$ models of $\Gamma$.
A triplet $(G, l, m)$ is said to be a {\it subdivision} of a triplet $(G^{\prime}, l^{\prime}, m^{\prime})$ and written as $(G, l, m) \succ (G^{\prime}, l^{\prime}, m^{\prime})$ if for any $e^{\prime} \in E(G^{\prime})$ and $e_i \in E(G)$ such that $e^{\prime} = e_1 \sqcup \cdots \sqcup e_n$, $m^{\prime}(e^{\prime})$ divides all $m(e_i)$.
In particular, if $m^{\prime}(e^{\prime})$ and all $m(e_i)$ equals, then $(G^{\prime}, l^{\prime}, m^{\prime})$ is said to be a {\it trivial} subdivision of $(G, l, m)$ and then $(G^{\prime}, l^{\prime}, m^{\prime})$ is denoted by $(G^{\prime}, l^{\prime}, m)$.
%A triplet $(G, l, m)$ is said to be a {\it subdivision} of a triplet $(G^{\prime}, l^{\prime}, m^{\prime})$ and written as $(G, l, m) \succ (G^{\prime}, l^{\prime}, m^{\prime})$ if for any $e^{\prime} \in E(G^{\prime})$ and $e_i \in E(G)$ such that $e^{\prime} = e_1 \sqcup \cdots \sqcup e_n$, $m^{\prime}(e^{\prime})$ divides all $m(e_i)$.
%In particular, if $m^{\prime}(e^{\prime})$ and all $m(e_i)$ equals, then $(G^{\prime}, l^{\prime}, m^{\prime})$ is said to be a {\it trivial} subdivision of $(G, l, m)$ and then $(G^{\prime}, l^{\prime}, m^{\prime})$ is denoted by $(G^{\prime}, l^{\prime}, m)$.
}
%%{\upshape
%%Let $\Gamma$ be a metric graph, and $(G, l) \succ (G^{\prime}, l^{\prime})$ models of $\Gamma$.
%%A triplet $(G, l, m_G)$ is said to be a {\it subdivision} of a triplet $(G^{\prime}, l^{\prime}, m_{G^{\prime}})$ and written as $(G, l) \succ (G^{\prime}, l^{\prime})$
%%if these satisfy the following condition:

%%for any $e^{\prime} = e_1 + \cdots + e_n \in E(G^{\prime})$, where each $e_i$ is in $E(G)$, $m_{G^{\prime}}(e^{\prime})$ divides all $m_G(e_i)$.

%%In particular, if $m_{G^{\prime}}$ is equal to all $m_G(e_i)$, then $(G, l, m_G)$ is said to be a {\it trivial} subdivision of $(G^{\prime}, l^{\prime}, m_{G^{\prime}})$.
%%}
\end{dfn}

\begin{dfn}
\upshape{
For a quadruplet $(\Gamma, G, l, m)$, the {\it metric graph with an edge-multiplicity}, denoted by $\Gamma_{m}$, is defined by the pair of metric graph $\Gamma$ and $m$ such that we can choose only models $(G^{\prime}, l^{\prime}) \prec (G, l)$ of $\Gamma$.
%Let $\Gamma_{m}$ be a quadruplet $(\Gamma, G, l, m)$.
%%Let $\Gamma_{m_G}$ be the quadruplet $(\Gamma, G, l, m_G)$.
%We call $\Gamma_{m}$ a {\it metric graph with an edge-multiplicity}.
The word ``a {\it point} $x$ on $\Gamma_{m}$'' means that $x \in \Gamma$.
The {\it genus} of $\Gamma_{m}$ is the genus of $\Gamma$.
%The {\it genus} of $\Gamma_{m}$ is the one of $\Gamma$.
}
\end{dfn}

%For a metric graph with an edge-multiplicity, we use same terms and notations for the underlying metric graph.

\begin{dfn}
\upshape{
Let $\Gamma_{m}$ be a metric graph with an edge-multiplicity.
${\rm Div}(\Gamma_{m})$ is defined by ${\rm Div}(\Gamma)$ and an element of ${\rm Div}(\Gamma_{m})$ is called a {\it divisor} on $\Gamma_{m}$.
The {\it canonical divisor} on $\Gamma_{m}$ is the canonical divisor on $\Gamma$.
We define ${\rm Rat}(\Gamma_{m})$ as ${\rm Rat}(\Gamma)$.
We call an element of ${\rm Rat}(\Gamma_{m})$ a {\it rational function} on $\Gamma_{m}$.
}
\end{dfn}

%\begin{dfn}
%\upshape{
%Let $\Gamma_{m}$ be a metric graph with an edge-multiplicity.
%${\rm Div}(\Gamma_{m})$ is defined by ${\rm Div}(\Gamma)$ and an element of ${\rm Div}(\Gamma_{m})$ is called a {\it divisor} on $\Gamma_{m}$.
%The {\it canonical divisor} on $\Gamma_{m}$ is the canonical divisor on $\Gamma$.
%${\rm Rat}(\Gamma_{m})$ is defined as the subset of ${\rm Rat}(\Gamma)$ whose elements have integral multiple slopes of $m(e)$ on each edge $e$ of $G$ or $- \infty$.
%We call an element of ${\rm Rat}(\Gamma_{m})$ a {\it rational function} on $\Gamma_{m}$.
%}
%\end{dfn}

%For a metric graph with an edge-multiplicity, we use same terms and notations for the underlying metric graph.

Note that for an edge $e$ of $G$ and $f \in {\rm Rat}(\Gamma_{m})$, $f$ has different finite slopes on $e$ since $f$ may have plural pieces.

For a metric graph with an edge-multiplicity, we use same terms and notations for the underlying metric graph.

\subsection{Harmonic morphisms with edge-multiplicities}

\begin{dfn}
\upshape{
Let $\Gamma_m, \Gamma^{\prime}_{m^{\prime}}$ be metric graphs with edge-multiplicities $m, m^{\prime}$, respectively, and $\varphi_{m,\, m^{\prime}} : \Gamma_m \rightarrow \Gamma^{\prime}_{m^{\prime}}$ be a continuous map.
The map $\varphi_{m,\, m^{\prime}}$ is called a {\it morphism} if $\varphi_{m,\, m^{\prime}}$ is a morphism as loopless models $(G, l)$ and $(G^{\prime}, l^{\prime})$.
%The map $\varphi_{m,\, m^{\prime}}$ is called a {\it morphism} if $\varphi_{m,\, m^{\prime}}$ is a morphism as loopless models and for corresponding loopless models $(G, l)$ and $(G^{\prime}, l^{\prime})$ and any edge $e$ of $G$, $m(e)$ devises $m^{\prime}(\varphi_{m,\, m^{\prime}}(e))$ if $\varphi_{m,\, m^{\prime}}(e)$ is an edge of $G^{\prime}$.
For an edge $e$ of $G$, if $\varphi_{m,\, m^{\prime}}(e)$ is a vertex of $G^{\prime}$, let $m^{\prime}(\varphi_{m,\, m^{\prime}}(e)) := 0$ formally.
%If $\varphi_{m,\, m^{\prime}}(e)$ is a vertex of $G^{\prime}$, let $m^{\prime}(\varphi_{m,\, m^{\prime}}(e)) := 0$ formally.
%there exist a model $(G, l)$ of $\Gamma_m$ and a model $(G^{\prime}, l^{\prime})$ of $\Gamma^{\prime}_{m^{\prime}}$ such that the image of the set of vertices of $G$ by $\varphi_{m,\, m^{\prime}}$ is a subset of the set of vertices of $G^{\prime}$, the inverse image of the relative interior of any edge of $G^{\prime}$ by $\varphi_{m,\, m^{\prime}}$ is the union of the relative interiors of a finite number of edges of $G$ and the restriction of $\varphi_{m,\, m^{\prime}}$ to any edge $e$ of $G$ is a dilation by
%Then, let 
%\[
%{\rm deg}^{m,\, m^{\prime}}_e(\varphi_{m,\, m^{\prime}}) := \frac{m^{\prime}(\varphi_{m,\, m^{\prime}}(e))}{m(e)} \cdot {\rm deg}_e(\varphi_{m,\, m^{\prime}}) = \frac{m^{\prime}(\varphi_{m,\, m^{\prime}}(e)) \cdot l^{\prime}(\varphi_{m,\, m^{\prime}}(e))} {m(e) \cdot l(e)} \in {\boldsymbol{Z}}_{\ge0}.
%\]
%, both being loopless,
%Note that the dilation factor on $e$ with ${\rm deg}_e(\varphi) \ne 0$ represents the ratio of the distance of the images of any two points $x$ and $y$ except points at infinity on $e$ to that of original $x$ and $y$.
%If an edge $e$ is mapped to a vertex of $G^{\prime}$ by $\varphi$, then ${\rm deg}_e(\varphi) = 0$.
The morphism $\varphi_{m,\, m^{\prime}}$ is said to be {\it finite} if $\varphi_{m,\, m^{\prime}}$ is finite as a morphism of loopless models.
%The morphism $\varphi_{m,\, m^{\prime}}$ is said to be {\it finite} if ${\rm deg}^{m,\, m^{\prime}}_e(\varphi_{m,\, m^{\prime}}) > 0$ for any edge $e$ of $G$.
%For any half-edge $h$ of any point on $\Gamma_m$, we define ${\rm deg}^{m,\, m^{\prime}}_h(\varphi_{m,\, m^{\prime}})$ as ${\rm deg}^{m,\, m^{\prime}}_e(\varphi_{m,\, m^{\prime}})$, where $e$ is the edge of $G$ containing $h$.
}
\end{dfn}

\begin{dfn}
\upshape{
Let $\varphi_{m,\, m^{\prime}} : \Gamma_m \rightarrow \Gamma^{\prime}_{m^{\prime}}$ be a morphism of metric graphs with edge-multiplicities.

Let $\Gamma^{\prime}_{m^{\prime}}$ be not a singleton and $x$ a point on $\Gamma_m$.
The morphism $\varphi_{m,\, m^{\prime}}$ is {\it harmonic at} $x$ if for any edge $e_1$ of $G$ adjacent to $x$, $m(e_1)$ devides $m^{\prime}(\varphi_{m,\, m^{\prime}}(e_1))$ and the number
%The morphism $\varphi_{m,\, m^{\prime}}$ is {\it harmonic at} $x$ if the number
\[
{\rm deg}^{m,\, m^{\prime}}_x(\varphi_{m,\, m^{\prime}}) := \sum_{x \in h \mapsto h^{\prime}} \frac{m^{\prime}(\varphi_{m,\, m^{\prime}}(e))}{m(e)} \cdot  {\rm deg}_h(\varphi_{m,\, m^{\prime}})
%{\rm deg}^{m,\, m^{\prime}}_x(\varphi_{m,\, m^{\prime}}) := \sum_{x \in h \mapsto h^{\prime}}{\rm deg}^{m,\, m^{\prime}}_h(\varphi_{m,\, m^{\prime}})
\]
is independent of the choice of half-edge $h^{\prime}$ emanating from $\varphi_{m,\, m^{\prime}}(x)$, where $h$ is a connected component of the inverse image of $h^{\prime}$ by $\varphi_{m,\, m^{\prime}}$ containing $x$ and $e$ is the edge of $G$ containing $h$.
%is a nonnegative integer and independent of the choice of half-edge $h^{\prime}$ emanating from $\varphi_{m,\, m^{\prime}}(x)$, where $h$ is a connected component of the inverse image of $h^{\prime}$ by $\varphi_{m,\, m^{\prime}}$ containing $x$ and $e$ is the edge of $G$ containing $h$.
%is independent of the choice of half-edge $h^{\prime}$ emanating from $\varphi_{m,\, m^{\prime}}(x)$, where $h$ is a connected component of the inverse image of $h^{\prime}$ by $\varphi_{m,\, m^{\prime}}$ containing $x$.
The morphism $\varphi_{m,\, m^{\prime}}$ is {\it harmonic} if it is harmonic at all points on $\Gamma_m$.
%One can check that if $\varphi$ is a finite harmonic morphism, then the number
For a point $x^{\prime}$ on $\Gamma^{\prime}_{m^{\prime}}$,
\[
{\rm deg}^{m,\, m^{\prime}}(\varphi_{m,\, m^{\prime}}) := \sum_{x \mapsto x^{\prime}}{\rm deg}^{m,\, m^{\prime}}_x(\varphi_{m,\, m^{\prime}})
\]
is said the {\it degree} of $\varphi_{m,\, m^{\prime}}$, where $x$ is an element of the inverse image of $x^{\prime}$ by $\varphi_{m,\, m^{\prime}}$.
%is independent of the choice of a point $x^{\prime}$ on $\Gamma^{\prime}$, and is said the {\it degree} of $\varphi$, where $x$ is an element of the inverse image of $x^{\prime}$ by $\varphi$.
If $\Gamma^{\prime}_{m^{\prime}}$ is a singleton and $\Gamma_m$ is not a singleton, for any point $x$ on $\Gamma_m$, 
%we define ${\rm deg}_x(\varphi)$ as arbitrary nonnegative integer and we regard $\varphi$ as a harmonic morphism with arbitrary degree.
we define ${\rm deg}^{m,\, m^{\prime}}_x(\varphi_{m,\, m^{\prime}})$ as zero so that we regard $\varphi_{m,\, m^{\prime}}$ as a harmonic morphism of degree zero.
If both $\Gamma_m$ and $\Gamma^{\prime}_{m^{\prime}}$ are singletons, we regard $\varphi_{m,\, m^{\prime}}$ as a harmonic morphism which can have any number of degree.
}
\end{dfn}

\begin{lemma}
$\sum_{x \mapsto x^{\prime}}{\rm deg}^{m,\, m^{\prime}}_x(\varphi_{m,\, m^{\prime}})$is independent of the choice of a point $x^{\prime}$ on $\Gamma^{\prime}$.
\end{lemma}

%\begin{lemma}
%$\sum_{x \mapsto x^{\prime}}{\rm deg}_x(\varphi_{m,\, m^{\prime}})$is independent of the choice of a point $x^{\prime}$ on $\Gamma^{\prime}$.
%\end{lemma}

\begin{proof}
It is sufficient to check that for any vertex of $G^{\prime}$, the sum is same.
Let $x_1^{\prime}$ and $x_2^{\prime}$ be vertices of $G^{\prime}$ both adjacent to an edge $e^{\prime}$ of $G^{\prime}$.
Let $h_1^{\prime}$ be the half-edge of $x_1^{\prime}$ contained in $e^{\prime}$.
%Let $h_1^{\prime}$ (resp. $h_2^{\prime}$) be the half-edge of $x_1^{\prime}$ (resp. $x_2^{\prime}$) contained in $e^{\prime}$.
Then
\begin{eqnarray*}
\sum_{x_1 \mapsto x_1^{\prime}}{\rm deg}^{m,\, m^{\prime}}_{x_1}(\varphi_{m,\, m^{\prime}}) &=& \sum_{x_1 \mapsto x_1^{\prime}} \left( \sum_{x_1 \in h_1 \mapsto h_1^{\prime}}{\rm deg}^{m,\, m^{\prime}}_{h_1}(\varphi_{m,\, m^{\prime}}) \right)\\
&=& \sum_{x_1 \mapsto x_1^{\prime}} \left( \sum_{x_1 \in e_1 \mapsto e^{\prime}}{\rm deg}^{m,\, m^{\prime}}_{e_1}(\varphi_{m,\, m^{\prime}}) \right)\\
&=& \sum_{e \mapsto e^{\prime}}{\rm deg}^{m,\, m^{\prime}}_e(\varphi_{m,\, m^{\prime}}).
\end{eqnarray*}
Similarly,
\[
\sum_{x_2 \mapsto x_2^{\prime}}{\rm deg}^{m,\, m^{\prime}}_{x_2}(\varphi_{m,\, m^{\prime}}) = \sum_{e \mapsto e^{\prime}}{\rm deg}^{m,\, m^{\prime}}_e(\varphi_{m,\, m^{\prime}}).\qedhere
\]
\end{proof}

%\begin{proof}
%It is sufficient to check that for any vertex of $G^{\prime}$, the sum is same.
%Let $x_1^{\prime}$ and $x_2^{\prime}$ be vertices of $G^{\prime}$ both adjacent to an edge $e^{\prime}$ of $G^{\prime}$.
%Let $h_1^{\prime}$ (resp. $h_2^{\prime}$) be the half-edge of $x_1^{\prime}$ (resp. $x_2^{\prime}$) contained in $e^{\prime}$.
%%Let $h_1^{\prime}$ (resp. $h_2^{\prime}$) be the half-edge of $x_1^{\prime}$ (resp. $x_2^{\prime}$) contained in $e^{\prime}$.
%Then
%\begin{eqnarray*}
%\sum_{x_1 \mapsto x_1^{\prime}}{\rm deg}_{x_1}(\varphi_{m,\, m^{\prime}}) &=& \sum_{x_1 \mapsto x_1^{\prime}} \left( \sum_{x_1 \in h_1 \mapsto h_1^{\prime}}{\rm deg}_{h_1}(\varphi_{m,\, m^{\prime}}) \right)\\
%&=& \sum_{x_1 \mapsto x_1^{\prime}} \left( \sum_{x_1 \in e_1 \mapsto e^{\prime}}{\rm deg}_{e_1}(\varphi_{m,\, m^{\prime}}) \right)\\
%&=& \sum_{e \mapsto e^{\prime}}{\rm deg}_e(\varphi_{m,\, m^{\prime}}).
%\end{eqnarray*}
%Similarly,
%\[
%\sum_{x_2 \mapsto x_2^{\prime}}{\rm deg}_{x_2}(\varphi_{m,\, m^{\prime}}) = \sum_{e \mapsto e^{\prime}}{\rm deg}_e(\varphi_{m,\, m^{\prime}}).\qedhere
%\]
%\end{proof}

The collection of metric graphs with edge-multiplicities together with harmonic morphisms between them forms a category.

\begin{dfn}
\upshape{
Let $\varphi_{m,\, m^{\prime}} : \Gamma_m \rightarrow \Gamma^{\prime}_{m^{\prime}}$ be a finite harmonic morphism of metric graphs with edge-multiplicities.
%Let $\varphi_{m,\, m^{\prime}} : \Gamma_m \rightarrow \Gamma^{\prime}_{m^{\prime}}$ be a finite harmonic morphism between metric graphs with edge-multiplicities.
For $f$ in ${\rm Rat}(\Gamma_m)$, the {\it push-forward} of $f$ is the function $(\varphi_{m,\, m^{\prime}})_\ast f: \Gamma^{\prime}_{m^{\prime}} \rightarrow \boldsymbol{R} \cup \{ \pm \infty \}$ defined by 
\[
(\varphi_{m,\, m^{\prime}})_\ast f(x^{\prime}) := \sum_{\substack{x \in \Gamma_{m} \\ \varphi_{m,\, m^{\prime}}(x) = x^{\prime}}} {\rm deg}^{m,\, m^{\prime}}_x(\varphi_{m,\, m^{\prime}}) \cdot f(x).
\]
The {\it pull-back} of $f^{\prime}$ in ${\rm Rat}(\Gamma^{\prime}_{m^{\prime}})$ is the function $(\varphi_{m,\, m^{\prime}})^{\ast}f^{\prime} : \Gamma_m \rightarrow \boldsymbol{R} ~\cup \{ \pm \infty \}$ defined by $(\varphi_{m,\, m^{\prime}})^{\ast}f^{\prime} := f^{\prime} \circ \varphi_{m,\, m^{\prime}}$.
We define the {\it push-forward homomorphism} on divisors $(\varphi_{m,\, m^{\prime}})_\ast : {\rm Div}(\Gamma_m) \rightarrow {\rm Div}(\Gamma^{\prime}_{m^{\prime}})$ by 
 homomorphism
\[
(\varphi_{m,\, m^{\prime}})_\ast (D) := \sum_{x \in \Gamma_m}D(x) \cdot \varphi_{m,\, m^{\prime}}(x).
\]
The {\it pull-back homomorphism} on divisors $(\varphi_{m,\, m^{\prime}})^{\ast}:{\rm Div}(\Gamma^{\prime}_{m^{\prime}}) \rightarrow {\rm Div}(\Gamma_m)$ is defined to be
\[
(\varphi_{m,\, m^{\prime}})^{\ast} (D^{\prime}) := \sum_{x \in \Gamma_m}{\rm deg}^{m,\, m^{\prime}}_x(\varphi_{m,\, m^{\prime}}) \cdot D^{\prime}(\varphi_{m,\, m^{\prime}}(x)) \cdot x.
\]
}
\end{dfn}

\begin{rem}
We need not to assume that $\varphi_{m,\, m^{\prime}}$ is finite to define pull-backs of rational functions and divisors.
\end{rem}

\begin{prop}
%\begin{lemma}
For any divisors $D$ on $\Gamma_m$ and $D^{\prime}$ on $\Gamma^{\prime}_{m^{\prime}}$, ${\rm deg}((\varphi_{m,\, m^{\prime}})_{\ast}(D)) = {\rm deg}(D)$ and ${\rm deg}((\varphi_{m,\, m^{\prime}})^{\ast}(D^{\prime})) = {\rm deg}^{m,\, m^{\prime}}(\varphi_{m,\, m^{\prime}})$ hold.
%\end{lemma}
\end{prop}

\begin{proof}
The first equation holds obviously.

Let $x^{\prime}$ be a point on $\Gamma^{\prime}_{m^{\prime}}$.
Since $\sum_{x \mapsto x^{\prime}}((\varphi_{m,\, m^{\prime}})^{\ast}(D^{\prime}))(x) = \sum_{x \mapsto x^{\prime}}{\rm deg}_x^{m,\, m^{\prime}}(\varphi_{m,\, m^{\prime}}) \cdot D^{\prime}(x^{\prime}) = {\rm deg}^{m,\, m^{\prime}}(\varphi_{m,\, m^{\prime}}) \cdot D^{\prime}(x^{\prime})$, we have the second equation.
\end{proof}

\begin{dfn}
\upshape{
Let $\varphi_{m,\, m^{\prime}} : \Gamma_m \rightarrow \Gamma^{\prime}_{m^{\prime}}$ be a finite harmonic morphism of metric graphs with edge-multiplicities.
For a rational function $f$ on $\Gamma_m$ other that $- \infty$, we define the number
\[
{\rm div}^{m,\, m^{\prime}}(f) := \sum_{x \in \Gamma_m} \left( \sum_{x \in e \in E(G)} \frac{m^{\prime}(\varphi_{m,\, m^{\prime}}(e))}{m(e)} \cdot \left(\text{the outgoing slope of } f \text{ on } e \text{ at } x \right) \cdot x \right)
\]
and call it the {\it principal divisor} with edge-multiplicities $m$ and $m^{\prime}$ defined by $f$.
}
\end{dfn}

\begin{prop}
%\begin{lemma}
For any rational functions $f$ on $\Gamma_m$ and $f^{\prime}$ on $\Gamma^{\prime}_{m^{\prime}}$ both other than $- \infty$, $(\varphi_{m,\, m^{\prime}})_{\ast}({\rm div}^{m,\, m^{\prime}} f ) = {\rm div}((\varphi_{m,\, m^{\prime}})_{\ast}(f))$ and $(\varphi_{m,\, m^{\prime}})({\rm div}(f^{\prime})) = {\rm div}^{m,\, m^{\prime}}((\varphi_{m,\, m^{\prime}})^{\ast}f^{\prime})$ hold.
%\end{lemma}
\end{prop}

\begin{proof}
Let us write $\varphi_{m,\, m^{\prime}}$ as $\varphi$ simply.
We may break $\Gamma_m$ and $\Gamma^{\prime}_{m^{\prime}}$ into sets $S$ and $S^{\prime}$ of segments along which $f$ and $\varphi_{\ast}f$, respectively, are linear and such that each segment $s \in S$ is mapped linearly to some $s^{\prime} \in S^{\prime}$.
Then at any point $x^{\prime}$ on $\Gamma^{\prime}_m$, we have
\begin{eqnarray*}
\varphi_{\ast}({\rm div}^{m,\, m^{\prime}} (f) ) (x^{\prime}) = \sum_{\substack{ x \in \Gamma_m \\ x \mapsto x^{\prime}}} {\rm div}^{m,\, m^{\prime}}(f) (x) = \sum_{x \in \varphi^{-1}(x^{\prime})} \sum_{s=xy \in S} \frac{m^{\prime}(\varphi(s))}{m(s)} \cdot \frac{f(y) - f(x)}{l(s)}
%(\varphi_{m,\, m^{\prime}})_{\ast}({\rm div}^{m,\, m^{\prime}} (f) ) (x^{\prime}) = \sum_{\substack{ x \in \Gamma_m \\ x \mapsto x^{\prime}}} {\rm div}^{m,\, m^{\prime}}(f) (x) = \sum_{x \in \varphi_{m,\, m^{\prime}}^{-1}(x^{\prime})} \sum_{s=xy \in S} \frac{m^{\prime}(\varphi_{m,\, m^{\prime}}(s))}{m(s)} \cdot \frac{f(y) - f(x)}{l(s)}
\end{eqnarray*}
and
\begin{eqnarray*}
&&{\rm div}(\varphi_{\ast} f)(x^{\prime})\\
&=& \sum_{s^{\prime} = x^{\prime} y^{\prime} \in S^{\prime}} \frac{(\varphi_{\ast} f) (y^{\prime}) - (\varphi_{\ast} f) (x^{\prime})} {l^{\prime}(s^{\prime})}\\
%&&{\rm div}((\varphi_{m,\, m^{\prime}})_{\ast} f)(x^{\prime})\\ &=& \sum_{s^{\prime} = x^{\prime} y^{\prime} \in S^{\prime}} \frac{((\varphi_{m,\, m^{\prime}})_{\ast} f) (y^{\prime}) - ((\varphi_{m,\, m^{\prime}})_{\ast} f) (x^{\prime})} {l^{\prime}(s^{\prime})}\\
&=& \sum_{s^{\prime} = x^{\prime} y^{\prime} \in S^{\prime}} \left\{ \sum_{\substack{ y \in \Gamma_m \\ \varphi(y) = y^{\prime}}} \left( \sum_{ y \in s \mapsto s^{\prime}} \frac{m^{\prime}(s^{\prime})}{m(s)} \cdot \frac{l^{\prime}(s^{\prime})}{l(s)} \right) f(y) - \sum_{\substack{ x \in \Gamma_m \\ \varphi(x) = x^{\prime}}} \left( \sum_{ x \in s \mapsto s^{\prime}} \frac{m^{\prime}(s^{\prime})}{m(s)} \cdot \frac{l^{\prime}(s^{\prime})}{l(s)} \right) f(x) \right\} \cdot \frac{1}{ l^{\prime}(s^{\prime})}\\
%&=& \sum_{s^{\prime} = x^{\prime} y^{\prime} \in S^{\prime}} \left\{ \sum_{\substack{ y \in \Gamma_m \\ \varphi_{m,\, m^{\prime}}(y) = y^{\prime}}} \left( \sum_{ y \in s \mapsto s^{\prime}} \frac{m^{\prime}(s^{\prime})}{m(s)} \cdot \frac{l^{\prime}(s^{\prime})}{l(s)} \right) f(y) - \sum_{\substack{ x \in \Gamma_m \\ \varphi_{m,\, m^{\prime}}(x) = x^{\prime}}} \left( \sum_{ x \in s \mapsto s^{\prime}} \frac{m^{\prime}(s^{\prime})}{m(s)} \cdot \frac{l^{\prime}(s^{\prime})}{l(s)} \right) f(x) \right\} \cdot \frac{1}{ l^{\prime}(s^{\prime})}\\
&=& \sum_{s^{\prime} = x^{\prime} y^{\prime} \in S^{\prime}} \left\{ \sum_{\substack{ y \in \Gamma_m \\ \varphi(y) = y^{\prime}}} \left( \sum_{ y \in s \mapsto s^{\prime}} \frac{m^{\prime}(s^{\prime})}{m(s)} \cdot \frac{1}{l(s)} \right) f(y) - \sum_{\substack{ x \in \Gamma_m \\ \varphi(x) = x^{\prime}}} \left( \sum_{ x \in s \mapsto s^{\prime}} \frac{m^{\prime}(s^{\prime})}{m(s)} \cdot \frac{1}{l(s)} \right) f(x) \right\}\\
%&=& \sum_{s^{\prime} = x^{\prime} y^{\prime} \in S^{\prime}} \left\{ \sum_{\substack{ y \in \Gamma_m \\ \varphi_{m,\, m^{\prime}}(y) = y^{\prime}}} \left( \sum_{ y \in s \mapsto s^{\prime}} \frac{m^{\prime}(s^{\prime})}{m(s)} \cdot \frac{1}{l(s)} \right) f(y) - \sum_{\substack{ x \in \Gamma_m \\ \varphi_{m,\, m^{\prime}}(x) = x^{\prime}}} \left( \sum_{ x \in s \mapsto s^{\prime}} \frac{m^{\prime}(s^{\prime})}{m(s)} \cdot \frac{1}{l(s)} \right) f(x) \right\}
&=& \sum_{s^{\prime} = x^{\prime} y^{\prime} \in S^{\prime}} \left\{ \sum_{\substack{ s = xy \in S \\ \varphi(s) = s^{\prime}}} \left( \frac{m^{\prime}(s^{\prime})}{m(s)} \cdot \frac{f(y)}{l(s)} - \frac{m^{\prime}(s^{\prime})}{m(s)} \cdot \frac{f(x)}{l(s)} \right) \right\}\\
&=& \sum_{x \in \varphi^{-1}(x^{\prime})} \sum_{s=xy \in S} \frac{m^{\prime}(\varphi(s))}{m(s)} \cdot \frac{f(y) - f(x)}{l(s)}.
\end{eqnarray*}

Let us assume that $\Gamma_m$ and $\Gamma^{\prime}_{m^{\prime}}$ are broken into $S_1$ and $S^{\prime}_1$ of segments along which $\varphi^{\ast}f^{\prime}$ and $f^{\prime}$, respectively, have the same conditions as that of $S$ and $S^{\prime}$.
Then for any point $x$ on $\Gamma_m$, we have
\begin{eqnarray*}
\left( \varphi^{\ast}({\rm div}(f^{\prime})) \right)(x) &=& {\rm deg}^{m,\, m^{\prime}}_x (\varphi_{m,\, m^{\prime}}) \cdot ({\rm div}(f^{\prime})(\varphi(x)))\\
&=& {\rm deg}^{m,\, m^{\prime}}_x (\varphi_{m,\, m^{\prime}}) \cdot \left( \sum_{s^{\prime} = \varphi(x) y^{\prime} \in S^{\prime}} \frac{f^{\prime}(y^{\prime}) - f^{\prime}(\varphi(x))} {l^{\prime}(s^{\prime})} \right)\\
&=& \sum_{s^{\prime} = \varphi(x) y^{\prime} \in S^{\prime}} {\rm deg}^{m,\, m^{\prime}}_x (\varphi_{m,\, m^{\prime}}) \cdot \frac{f^{\prime}(y^{\prime}) - f^{\prime}(\varphi(x))} {l^{\prime}(s^{\prime})}\\
&=& \sum_{s^{\prime} = \varphi(x) y^{\prime} \in S^{\prime}} \sum_{s = xy \mapsto s^{\prime}}\frac{m^{\prime}(s^{\prime})}{m(s)} \cdot \frac{l^{\prime}(s^{\prime})}{l(s)} \cdot \frac{f^{\prime}(y^{\prime}) - f^{\prime}(\varphi(x))} {l^{\prime}(s^{\prime})}\\
&=& \sum_{s^{\prime} = \varphi(x) y^{\prime} \in S^{\prime}} \sum_{s = xy \mapsto s^{\prime}}\frac{m^{\prime}(s^{\prime})}{m(s)} \cdot \frac{f^{\prime}(y^{\prime}) - f^{\prime}(\varphi(x))} {l(s)}\\
&=& \sum_{s = xy \in S}\frac{m^{\prime}(s^{\prime})}{m(s)} \cdot \frac{f^{\prime}(\varphi(y)) - f^{\prime}(\varphi(x))} {l(s)}\\
&=& \sum_{s = xy \in S}\frac{m^{\prime}(s^{\prime})}{m(s)} \cdot \frac{(\varphi^{\ast}f^{\prime})(y) - (\varphi^{\ast}f^{\prime})(x)} {l(s)}\\
&=& ({\rm div}^{m,\, m^{\prime}}(\varphi^{\ast}(f^{\prime})))(x). \qedhere
\end{eqnarray*}
\end{proof}

%\begin{lemma}

%\end{lemma}

%\begin{lemma}

%\end{lemma}

%\begin{rem}
%It might be that ${\rm deg}((\varphi_{m,\, m^{\prime}})_\ast(D)) \not= {\rm deg}(D)$ , ${\rm deg}((\varphi_{m,\, m^{\prime}})^{\ast}(D^{\prime})) \not= {\rm deg}(\varphi_{m,\, m^{\prime}}) \cdot {\rm deg}(D^{\prime})$, $(\varphi_{m,\, m^{\prime}})_\ast ({\rm div}(f)) \not= {\rm div}((\varphi_{m,\, m^{\prime}})_\ast f)$ and $(\varphi_{m,\, m^{\prime}})^{\ast} ({\rm div}(f^{\prime})) \not= {\rm div}((\varphi_{m,\, m^{\prime}})^{\ast} f^{\prime})$ for some divisors $D$, $D^{\prime}$ on $\Gamma_m$, $\Gamma^{\prime}_{m^{\prime}}$, some $f$ in ${\rm Rat}(\Gamma_m)^{\times}$ and some $f^{\prime}$ in ${\rm Rat}(\Gamma^{\prime}_{m^{\prime}})^{\times}$, respectively.
%However when both $m$ and $m^{\prime}$ is trivial, all equalities hold.
%\end{rem}

%One can check that ${\rm deg}((\varphi_{m,\, m^{\prime}})_\ast(D)) = {\rm deg}(D)$ , ${\rm deg}((\varphi_{m,\, m^{\prime}})^{\ast}(D^{\prime})) = {\rm deg}(\varphi_{m,\, m^{\prime}}) \cdot {\rm deg}(D^{\prime})$, $(\varphi_{m,\, m^{\prime}})_\ast ({\rm div}(f)) = {\rm div}((\varphi_{m,\, m^{\prime}})_\ast f)$ and $(\varphi_{m,\, m^{\prime}})^{\ast} ({\rm div}(f^{\prime})) = {\rm div}((\varphi_{m,\, m^{\prime}})^{\ast} f^{\prime})$ for any divisors $D$, $D^{\prime}$ on $\Gamma_m$, $\Gamma^{\prime}_{m^{\prime}}$, any $f$ in ${\rm Rat}(\Gamma_m)^{\times}$ and any $f^{\prime}$ in ${\rm Rat}(\Gamma^{\prime}_{m^{\prime}})^{\times}$, respectively.

\begin{dfn}
\upshape{
Let $\varphi_{m,\, m^{\prime}} : \Gamma_m \rightarrow \Gamma^{\prime}_{m^{\prime}}$ be a map between metric graphs with edge-multiplicities $m$ and $m^{\prime}$ and let $K$ be a finite group.
%Let $\varphi_{m,\, m^{\prime}} : \Gamma_m \rightarrow \Gamma^{\prime}_{m^{\prime}}$ be a finite harmonic morphism of metric graphs with edge-multiplicities $m$ and $m^{\prime}$ and let $K$ be a finite group.
$\varphi_{m,\, m^{\prime}}$ is a {\it $K$-Galois covering} on $\Gamma^{\prime}_{m^{\prime}}$ if $\varphi_{m,\, m^{\prime}}$ is a finite harmonic morphism of metric graphs with edge-multiplicities, $|K| = {\rm deg}^{m,\, m^{\prime}}(\varphi_{m,\, m^{\prime}})$ and $K$ acts on transitively on every fibre and $K$ keeps edge-multiplicities.
}
\end{dfn}

\begin{rem}
\upshape{
If $\varphi_{m,\, m^{\prime}} : \Gamma_m \rightarrow \Gamma^{\prime}_{m^{\prime}}$ is $K$-Galois, then there exists a group homomorphism $K \rightarrow {\rm Isom}_{(G, l, m)}$ for a model $(G, l)$ of $\Gamma_m$.
%If $\varphi_{m,\, m^{\prime}} : \Gamma_m \rightarrow \Gamma^{\prime}_{m^{\prime}}$ is $K$-Galois, then there exists a group homomorphism $K \rightarrow {\rm Aut}_{(G, l, m)}$ for a model $(G, l)$ of $\Gamma_m$.
}
\end{rem}

\end{document}